\documentclass{amsart}

\usepackage[margin=1.10in]{geometry}
\usepackage{microtype}
\usepackage{hyperref}
\usepackage{amssymb}
\usepackage{mathtools}
\usepackage{aliascnt}

\usepackage[nameinlink]{cleveref}

\usepackage[style=alphabetic,maxalphanames=99,maxbibnames=99,url=false,isbn=false]{biblatex}
\newtheorem{thm}{Theorem}[section]
\newaliascnt{cor}{thm}
\newtheorem{cor}[cor]{Corollary}
\aliascntresetthe{cor}
\newaliascnt{lem}{thm}
\newtheorem{lem}[lem]{Lemma}
\aliascntresetthe{lem}
\newaliascnt{prop}{thm}
\newtheorem{prop}[prop]{Proposition}
\aliascntresetthe{prop}
\newaliascnt{conj}{thm}
\newtheorem{conj}[conj]{Conjecture}
\aliascntresetthe{conj}

\newtheorem*{thmA}{Theorem A}
\newtheorem*{thmB}{Theorem B}
\newtheorem*{thmC}{Theorem C}
\newtheorem*{thmD}{Theorem D}
\newtheorem*{thmE}{Theorem E}
\newtheorem*{conjF}{Conjecture F}

\theoremstyle{definition}
\newaliascnt{defn}{thm}
\newtheorem{defn}[defn]{Definition}
\aliascntresetthe{defn}
\newaliascnt{remark}{thm}
\newtheorem{remark}[remark]{Remark}
\aliascntresetthe{remark}

\crefalias{thm}{theorem}
\crefalias{cor}{corollary}
\crefalias{lem}{lemma}
\crefalias{prop}{proposition}
\crefalias{defn}{definition}
\crefalias{eg}{example}
\crefalias{qn}{question}
\crefalias{claim}{claim}
\crefalias{remark}{remark}
\crefalias{conj}{conjecture}

\crefname{question}{question}{questions}
\Crefname{question}{Question}{Questions}
\crefname{claim}{claim}{claims}
\Crefname{claim}{Claim}{Claims}
\crefname{remark}{remark}{remarks}
\Crefname{remark}{Remark}{Remarks}
\crefname{conjecture}{conjecture}{conjectures}
\Crefname{conjecture}{Conjecture}{Conjectures}

\newcommand{\N}{\mathbb{N}}

\newcommand{\ve}{\varepsilon}
\newcommand{\Z}{\mathbb{Z}}

\newcommand{\symdiff}{\mathbin{\triangle}}

\DeclareMathOperator{\asdim}{asdim}
\DeclareMathOperator{\sdim}{sdim}
\DeclareMathOperator{\udim}{udim}
\DeclareMathOperator{\diam}{diam}

\DeclareMathOperator{\supp}{supp}
\DeclareMathOperator{\vol}{vol}

\newcommand{\trivial}[1]{}
\newcommand{\define}[1]{\textbf{#1}} %definitions in bold
\renewcommand{\subset}{\subseteq}

\newcommand{\calA}{\mathcal{A}}
\newcommand{\calB}{\mathcal{B}}
\newcommand{\calC}{\mathcal{C}}
\newcommand{\calF}{\mathcal{F}}
\newcommand{\calU}{\mathcal{U}}
\newcommand{\calV}{\mathcal{V}}
\newcommand{\calW}{\mathcal{W}}

\newcommand{\actson}{\curvearrowright}

\newcommand{\powset}{\mathcal{P}} %powerset
\DeclareMathOperator{\asi}{asi} 
\DeclareMathOperator{\si}{si}

\newcommand{\union}{\cup}
\newcommand{\bigunion}{\bigcup}
\newcommand{\inters}{\cap}
\newcommand{\bndry}{\partial}

\title{Borel dimension growth and hyperfiniteness}

\author{Jan Grebík}
\address{Computer Science Institute, Charles University, Malostranské nám. 25, 118 00 Praha 1, Czechia}
\email{grebikj@gmail.com}
\thanks{JG and VR were supported by the Czech Science Foundation (GA \v{C}R), project No. 26-23599M}

\author{Andrew S. Marks}
\address{Department of Mathematics, UC Berkeley, 970 Evans Hall MC 3840, Berkeley CA 94720}
\email{marks@math.berkeley.edu}
\thanks{AM was supported in part by NSF grant DMS-2348788}

\author{Václav Rozhoň}
\address{Computer Science Institute, Charles University, Malostranské nám. 25, 118 00 Praha 1, Czechia}
\email{vaclav.rozhon@matfyz.cuni.cz}

\author{Forte Shinko}
\address{Department of Mathematics, UC San Diego, 9500 Gilman Drive, La Jolla, CA 92093}
\email{fshinko@ucsd.edu}

\date{\today}

\begin{document}

\begin{abstract}
  We show that every increasing union of Borel graphs of pointwise volume growth
  at most $\exp(O(r^\gamma))$ is hyperfinite, where $\gamma \approx 0.1523$ is
  the unique real root of $(1 - \gamma)^3 - 4\gamma = 0$. This implies a
  positive answer to the special case of Weiss's question on the hyperfiniteness
  of Borel actions of countable amenable groups, for countable amenable groups
  locally of volume growth $\exp(O(r^\gamma))$ for $\gamma$ as above. 
  % making partial progress on the question of Conley, Jackson, Marks, Seward,
  % and Tucker-Drob of whether every Borel graph of subexponential growth is
  % hyperfinite.
  We also show that every bounded degree Borel graph of subexponential volume growth
  has Borel F{\o}lner tilings, improving a result of Downarowicz and Zhang for
  graphs generated by free Borel actions of groups of subexponential volume
  growth. Our main tools are the development of a Borel analogue of the
  two-parameter dimension growth function of a metric space as introduced by
  Dranishnikov and Sapir, and ball carving algorithms from theoretical computer
  science. We end with a pair of conjectures relating Borel amenability, Borel
  subexponential dimension growth, and hyperfiniteness, which would imply a
  positive answer to Weiss's question.
\end{abstract}

\maketitle

\section{Introduction}

This paper is a contribution to the study of amenability, hyperfiniteness,
and almost finiteness of countable Borel equivalence relations. The study
of Borel reducibility of equivalence relations arose in descriptive set theory
as a way of studying the relative difficulty of classification problems in
mathematics. By the Glimm-Effros dichotomy of Harrington, Kechris, and
Louveau \cites{HKL}, the simplest nontrivial class of
Borel equivalence relations is the class of hyperfinite Borel equivalence
relations. However, many fundamental questions about hyperfinite Borel
equivalence relations remain open. One important open problem is Weiss's
question from 1984 of whether every Borel action of a countable amenable
group generates a hyperfinite orbit equivalence relation \cite{Weiss}. That is,
does group theoretic tameness exactly correspond to descriptive set
theoretic-simplicity? 

Regular progress has been made on Weiss's problem over the last several decades
for various classes of countable amenable groups: the groups $\Z^n$ (Weiss,
unpublished), finitely generated groups of polynomial volume growth
\cite{JKL}, abelian groups \cite{GJ}, locally nilpotent groups \cite{SS}, and
polycyclic groups \cite{CJMSTD}.

The work in the papers \cite{GJ} and \cite{SS} was quite technical
and required a great deal of information about the internal
geometry of the groups in question. This style of attack seemed ill suited to
understanding Weiss's problem in general, since we have a very poor
understanding of the geometry of arbitrary amenable groups.
In the past few years, some effort has been made to use softer methods
for approaching this problem. An advance here was made in
the paper \cite{CJMSTD} which introduced tools from the study of asymptotic
dimension to attack the problem, and greatly simplified those earlier
results at the same time as generalizing them. 

Another line of progress has come from trying to generalize these
techniques for proving hyperfiniteness of group actions to proving hyperfiniteness
of Borel graphs with various growth bounds or amenability properties. This
immediately forces one to confront the problem discussed above since
arbitrary graphs have no nice or ``regular'' geometry. 
In \cite[page 879]{CJMSTD20} it is mentioned as an open question whether every
Borel graph of uniformly subexponential growth is hyperfinite, and 
\cite[Problem 3.5]{M} asked whether every Borel graph of
uniformly bounded polynomial growth is hyperfinite. 
We note here that Tessera
\cite[Proposition 3.5]{T} showed in 2008 that every Borel graph of uniformly
subexponential growth on a standard
probability space $(X,\mu)$ is amenable and hence
hyperfinite modulo a $\mu$-null set by the Connes-Feldman-Weiss theorem
\cite{CFW} (see also \cite[Theorem 18]{Kast25}). So the measure-theoretic analogue
of the above problem on subexponential growth has a positive answer, analogously
to how the Ornstein-Weiss theorem \cite{OW} gives a positive measure-theoretic answer to
Weiss's problem.

% We quote from the paper \cite{M}:
% ``Finding techniques for resolving Problem 3.5 [whether Borel graphs of
% polynomial volume growth are hyperfinite] where there is far less regular
% geometric structure would be one way of making progress towards resolving
% Weiss's question in general since we know little about any regular geometric
% structure in arbitrary amenable groups''.

Recently, Bernshteyn and Yu \cite{BY} made a significant advance here,
showing that every Borel graph of polynomial growth is hyperfinite. Their result
uses techniques from local algorithms in theoretical computer science and
a Borel version of the Lov\'asz Local Lemma to do a randomized ball carving
construction to find a type of ``padded decomposition'' of the space, witnessing
that it has finite Borel asymptotic dimension. Hyperfiniteness then follows
from \cite{CJMSTD}. This use of 
the theory of distributed algorithms 
in descriptive set theory is one of many such fruitful connections that
have emerged in the past few years between these fields
growing out of the work in \cite{Ber23}.

Our main result is a generalization of the result of Bernshteyn and Yu,
showing that increasing unions of Borel graphs of sufficiently slow intermediate
growth are hyperfinite:

\begin{thmA}
Let $\gamma$ be the unique real root of $(1 - \gamma)^3 - 4
\gamma = 0$, so $\gamma \approx 0.1523$.
Suppose $X$ is a standard Borel space and
$\rho_0 \geq \rho_1 \geq \ldots$ is a decreasing
sequence of Borel extended metrics on $X$ so that for every
$n$, $\rho_n$ has pointwise 
volume growth $\exp(O(r^{\gamma}))$. That is, for every $n \in \N$ and $x \in
X$, there exist $C,N > 0$ such that $|B^{\rho_n}_r(x)| \leq
\exp(C r^\gamma)$ for all $r \geq N$. Then $\bigunion_n E_{\rho_n}$ is
hyperfinite, where $E_{\rho_n}$ is the finite distance equivalence relation for
$\rho_n$.
\end{thmA}
 
Since all the tools in our proofs are from dimension theory and metric geometry, we
have stated our results in general for Borel extended metric spaces. One important example of such a
decreasing sequence of Borel extended metrics $\rho_0 \geq \rho_1 \geq \ldots$ is where $G_0 \subset G_1 \subset \ldots$
is an increasing sequence of locally countable Borel graphs on a standard Borel
space $X$, and $\rho_n$ is the
path length metric on $G_n$. Hence, by applying Theorem A to these graph
metrics, the theorem shows that 
an increasing union of Borel graphs, each of pointwise volume growth
$\exp(O(r^{\gamma}))$, is hyperfinite. 

Another important example of a locally finite Borel extended metric space arises
when $\Gamma \actson X$ is a Borel action of a
countable group $\Gamma$ on a standard Borel space $X$. If $\Gamma$ has a finite symmetric generating set $S$, then 
the word metric $\rho_S$ on the action is a Borel extended metric, where $\rho_S(x,x')$ is the least
length of a word $w$ in the generators $S$ so that $w \cdot x = x'$. The 
finite distance equivalence relation for $\rho_S$ is the orbit equivalence relation of the action
$\Gamma \actson X$. More
generally, if $\Gamma$ is not finitely generated, let   
 $S_0 \subset S_1 \subset \ldots \subset \Gamma$ be a sequence of finite symmetric sets
 so that $\bigunion_n S_n$ generates $\Gamma$. Then the associated locally
 finite word metrics form a decreasing sequence $\rho_{S_0} \geq \rho_{S_1} \geq
 \ldots$.
By applying Theorem
A to such metrics, this implies that Weiss's
problem has a positive answer for any countable group whose finitely generated subgroups all have
volume growth $\exp(O(r^{\gamma}))$. 
However, we note that the existence of groups of that growth which are not
of locally polynomial growth (and hence already known to have hyperfinite
actions by \cite{SS}) is an open problem.
In particular, Grigorchuk's gap conjecture
\cite{Grig91}
states that any group of growth slower than  
$\exp(r^{0.5})$ has polynomial growth. The group with the currently known slowest
volume growth is the first Grigorchuk group, for which
Erschler and Zheng proved that its volume growth is $\exp(r^{\alpha+o(1)})$
where $\alpha \approx 0.7674$ \cite{EZ}.
Thus, it remains an interesting open
problem to increase the exponent $\gamma$ in Theorem A. 
However, we do note that there are many interesting Borel graphs and non-free
Borel actions of countable groups of intermediate volume growth bounded by
$\exp(r^{\alpha})$ for $\alpha \leq \gamma$, which are therefore hyperfinite by
Theorem A.
See, for example,
\cite{KW,BCSTDN}.

A problem related to hyperfiniteness has been the study of F{\o}lner tilings or
almost finiteness \cite{E18} of graphs and metric spaces. Recall that a
finite set $A
\subset X$ in an
extended metric space $(X,\rho)$ is $(r,\ve)$-F{\o}lner if $|B_r(A)| \leq (1 +
\ve)|A|$, where $B_r(A)$ is the closed ball of radius $r$ around $A$. An
$(r,\ve)$-F{\o}lner tiling of $X$ is a partition of the space $X$ into $(r,\ve)$-F{\o}lner
sets of uniformly bounded diameter. Recently Downarowicz, Huczek, and Zhang showed the existence of such
tilings of finitely generated amenable groups \cite{DHZ}, and \cite{CJKMSTD}
studied measurable tilings of free actions of amenable groups. 
Using our techniques we also 
generalize a theorem of Downarowicz and Zhang \cite{DZ} (see also \cite{BBW}), who showed that every
free Borel action of a group of subexponential growth has a Borel F{\o}lner
tiling\footnote{The Borel result follows from their topological result for free continuous actions
on zero-dimensional spaces since given a Borel action of
a countable group on a Polish space $X$, one can change the topology of $X$ to a finer zero-dimensional Polish topology with the same
collection of Borel sets, where this action is continuous. See 
\cite[Exercise 13.5]{K95} and the discussion in \cite{BBW}}.
Our result generalizes this theorem both to non-free actions of groups of
subexponential growth and, more
generally, to extended metric spaces of subexponential growth:

\begin{thmB}
  Suppose $(X,\rho)$ is a Borel extended metric space of uniformly subexponential volume growth $\vol(r) \in e^{o(r)}$. Then for every $r,\ve > 0$, there is a Borel $(r,\ve)$-F{\o}lner tiling of $(X,\rho)$.
\end{thmB}

Our proofs use two main new tools. First, following \cite{BY}, we use
ideas from ball-carving
algorithms in theoretical computer science to find ways of efficiently subdividing a space
of subexponential growth into finite pieces. However, we do a simpler
deterministic ball-carving construction which does not require any use of the Lov\'asz Local Lemma. These ideas are contained in
\Cref{sec:ball-carving}, where we introduce these ideas in an abstract
form which we call \define{F{\o}lner layerings}, and show they exist in all Borel
metric spaces of subexponential growth.
The existence of such
F{\o}lner layerings can already be used to easily prove our Theorem B on
F{\o}lner tilings, and we construct these tilings in 
\Cref{sec:folner-tiling}. The existence
of Borel 
F{\o}lner layerings is in fact equivalent to the existence of
Borel F{\o}lner tilings for F{\o}lner Borel extended metric spaces, and we show
this in \Cref{subsec:tiling-layer-equiv}. 
However, these F{\o}lner layerings are easier to construct inductively and to
analyze.

Our second main tool is to generalize some techniques
from the study of finite Borel asymptotic dimension and its connection to
hyperfiniteness. The one-parameter dimension growth of a metric space--the function giving its dimension at each
``scale''--originated in work
of Gromov \cite{G99} and has been well studied. We introduce a Borel analogue of Dranishnikov
and Sapir's two-parameter dimension growth function $\dim^X(s,d)$ from their paper
\cite{DSX}. Roughly, the Borel dimension $\dim_B^X(s,d)$ of $(X,\rho)$
%with separation $s$ and diameter bound $d$ 
is the least $n$ so there is a cover
of $X$ by $n + 1$ families of Borel sets $\calU_0, \ldots, \calU_n$ of diameter at
most $d$ so that any two distinct sets in $\calU_i$ have distance strictly
greater than $s$ (see \Cref{defn:dim-growth}). So the limit of this Borel dimension
function gives Borel asymptotic dimension of $X$: 
\[\asdim_B(X) = \lim_{s \to
\infty} \lim_{d \to \infty} \dim_B^X(s,d).\]
We show that 
certain conditions on the growth of this Borel dimension function
$\dim_B^X(s,d)$
imply hyperfiniteness. More generally, we give a hyperfiniteness condition for
a countable sequence of decreasing Borel extended metrics, generalizing the work from
\cite[Theorem 7.3]{CJMSTD} on the hyperfiniteness of increasing unions of
equivalence relations of finite asymptotic dimension: 

\begin{thmC}
  Suppose $X$ is a standard Borel space, and $\rho_0 \geq \rho_1 \geq \ldots$ is
  a decreasing sequence of locally finite Borel extended metrics on $X$.
  Suppose there exist sequences $(s_n)_{n \in \N}$, $(d_n)_{n \in \N}$ of
  positive real numbers with $s_n \to \infty$ such that for every $m$, for all
  sufficiently large $n$, 
  \[
  24\dim_B^{\rho_{m}}(s_{n+1},d_{n+1})d_{n-1} \leq s_n \leq d_n. 
  \]
  Then $\bigunion_n E_{\rho_n}$ is hyperfinite.

  In particular, suppose $a\geq 1$ and $b \geq 0$ and for every $m$, there is a
  constant $C_m$ such that 
  \[
    \dim^{\rho_m}_B(s,C_m s^a)\in O(s^b).
  \]
  Then $\bigunion_n E_{\rho_n}$ is hyperfinite provided $ab\leq 1/4$.
\end{thmC}

We note that there is a reasonable definition of the Borel dimension also for
locally countable Borel extended metric spaces $X$, using covers where
the sets in $\calU_i$ may be infinite, but they must still have bounded diameter
and the associated equivalence relation $E_{\calU_i}$ whose classes are the
elements of $\calU_i$ must be smooth. Using that
definition of Borel dimension, Theorem C above also remains true for locally
countable Borel extended metric spaces. See \Cref{rem:locally-countable}.

In \Cref{sec:ball-carving} we use our results on 
F{\o}lner layerings to compute
the following upper bound on the Borel dimension of a space of volume growth
$\exp(O(r^{\alpha}))$:

\begin{thmD}
  Suppose $X$ is a Borel extended metric space of volume growth bounded by 
  \[\vol(r) \in \exp(O(r^{\alpha}))\]
  for some $\alpha \in (0,1)$.
  Then there exists a constant $C > 0$ such that 
  \[\dim^X_B(s, Cs^{1/(1-\alpha)}) \in O(s^{\alpha/(1-\alpha)^2}).\]
\end{thmD}

Together, Theorems C and D imply Theorem A for spaces with uniformly bounded
volume growth $\exp(O(r^\gamma))$. We set $a = 1/(1-\alpha)$ and $b =
\alpha/(1-\alpha)^2$ from the bounds in Theorem D, and then the condition $ab \leq
1/4$ from Theorem C gives hyperfiniteness for all $\alpha \leq \gamma$ where
$\gamma$ is as in Theorem A. The stronger version of Theorem A with only
pointwise volume growth bounds then follows from the uniformly bounded case by
defining a new collection of metrics with uniformly bounded volume growth whose finite
distance relations have the same union. 

We note that recently Jing Yu has announced an improvement of Theorem D
showing that under the same assumptions, $\dim^X_B(s, Cs^{1/(1-\alpha)}) \in
O(s^{\alpha/(1-\alpha)})$. The proof uses the Lov\'asz Local Lemma and randomized
ball carving methods from \cite{BY}. By combining this result with Theorem C as above,
Yu's result improves the bound on $\gamma$ in Theorem A to be the least
root of $4\gamma = (1 - \gamma)^2$ so $\gamma = 3 -
2 \sqrt{2} \approx 0.1716$. 

We remark that one fruitful aspect of our approach here is that it solves a
difficult issue in earlier work on Weiss's problem. Several previous partial
results on Weiss's problem, particularly \cite{GJ} and \cite{CJMSTD}, proved
those results for particular classes of countable amenable groups by first
proving positive results for free actions of these groups, then generalizing 
this to all actions by using tricks around classifying possible stabilizers
and the complexity of the conjugacy equivalence relations of these groups
to handle the non-free case. In particular, \cite{CJMSTD} uses a theorem of
Schneider and Seward \cite[Theorem 5.1]{SS} that if all free Borel actions of
countable polycyclic groups have hyperfinite orbit equivalence relations, then
all Borel actions of polycyclic groups have hyperfinite orbit equivalence
relations to overcome this difficulty.
Indeed, one result in \cite{CJMSTD}
is that free Borel actions of
the lamplighter group $\Z_2 \wr \Z$ are hyperfinite, but it is
left open whether all (in particular non-free) Borel actions of that group are hyperfinite. As described in
\cite[Section 8]{SS}, this difficulty in handling non-freeness is a
significant issue in continuing work on Weiss's problem. However, this
issue is handled easily in the present work. Because we are only using
upper bounds on growth as our assumption, and this upper bound
passes to subspaces and to quotient groups, our techniques work for all Borel
actions of groups with bounded growth of the type we are considering,
whether or not the action is free. We hope the ideas here continue to be
useful in approaching the non-free case of Weiss's problem. 

In \Cref{sec:mu-hyperfiniteness}, we discuss how dimension growth is related to
measure-theoretic hyperfiniteness. We introduce the Borel separation index
function $\si^X_B(s)$ of a Borel extended metric space $X$, as a quantitative
version of the 
asymptotic separation index of \cite{CJMSTD}. Precisely, $\si^X_B(s)$ is the least
$n$ so that there is a cover
of $X$ by $n + 1$ families of Borel sets $\calU_0, \ldots, \calU_n$ all of
finite diameter (but with no uniform bound on their diameter)
so that any two distinct sets in $\calU_i$ have distance strictly
greater than $s$. So the limit of this Borel separation index is the asymptotic
separation index from \cite{CJMSTD}: $\asi_B(X) = \lim_{s \to \infty} \si^X_B(s)$.
If $\mu$ is a Borel
probability measure on $X$, we also define the $\mu$-measurable separation
index $\si^X_\mu(s)$ where we may discard an invariant nullset, and
$\asi_\mu(X) = \lim_{s \to \infty} \si^X_\mu(s)$.
We then formulate a metric analogue of
Elek's \cite{E12} result (refining 
\cite{Kai}) characterizing $\mu$-hyperfiniteness of Borel graphs, see \Cref{lem:mu-hyperfinite-equiv}. We then use this 
to show Theorem E from the introduction, refining a result of Weilacher
\cite{Weil}
who proved the equivalence of (1) and (2) below, building on work of Conley and
Miller \cite{CM16}:

\begin{thmE}
  If $(X,\rho)$ is a locally finite Borel extended metric space, and $\mu$ is any Borel
  probability measure on $X$, then the following are equivalent:
  \begin{enumerate}
  \item $E_\rho$ is $\mu$-hyperfinite
  \item $\asi_\mu(X) \leq 1$.
  \item $\si_\mu^X$ has non-exponential growth: there is no $b > 1$ such that
  $\si_\mu^X(s) \geq b^s$ for all sufficiently large $s$.
  \end{enumerate}
\end{thmE}

Finally, we make some conjectures concerning a Borel version of the single variable
Borel dimension growth function $\sdim^X_B(s) = \lim_{d
\to \infty} \dim^X_B(s,d)$ which was originally introduced by
Gromov \cite{G99}.
We will call this single-variable Borel dimension growth function
$\sdim^X_B(s)$ \define{the Borel
dimension growth of $X$}.
As a corollary of Theorem E, if every measure hyperfinite Borel
equivalence relation is hyperfinite, then every locally finite Borel extended metric space
of subexponential Borel dimension growth is hyperfinite since $\si^X_\mu(s) \leq
\si^X_B(s) \leq \sdim^X_B(s)$ for all Borel probability measures $\mu$ on $X$. Since it
is an open problem whether every measure-hyperfinite Borel equivalence
relation is hyperfinite, we see there is a large gap between the
current rates of dimension growth which imply hyperfiniteness, and the
exponential dimension growth rates which are present in all known
non-hyperfinite CBERs which are not measure hyperfinite. 

We end the paper with a pair of conjectures on Borel dimension growth which
would imply a positive solution to Weiss's problem. 

\begin{conjF}
  \mbox{ }
  \begin{enumerate}
    \item 
    Suppose $\Gamma \actson X$ is a Borel action of a finitely
    generated amenable group $\Gamma$ %with finite symmetric generating set $S$
    on a standard Borel space $X$, and $\rho$ %$\rho_S$
    is the word metric on this action.
    Then $(X,\rho)$ has subexponential Borel dimension growth.
    \item If $(X,\rho)$ is a locally finite Borel extended metric space which 
    has subexponential Borel dimension growth, then $E_\rho$ is hyperfinite. 
  \end{enumerate}
\end{conjF}

The following considerations motivate Conjecture F and provide partial evidence
for its two assertions. First, part (1) is suggested by the conjectured
equivalence between subexponential dimension growth and property A in the
classical setting (see \cite[Question 2.13]{DSX} and \cite[Page 922]{O14}).
Second, Elek and Tim\'ar's work on F{\o}lner property A, strong F{\o}lner
hyperfiniteness, and uniform Borel amenability provides related structural evidence \cite{ET25,ETX}. Third, the
close connection between measure-theoretic hyperfiniteness and subexponential
separation-index growth in Theorem E provides evidence for part (2). Finally,
several results in \cite{CJMSTD} and in the present paper establish special cases
of the two assertions, both by proving Borel dimension-growth bounds for actions
of certain amenable groups and by showing that sufficiently slow dimension
growth implies hyperfiniteness.

Part (1) of Conjecture F would already require substantial progress in the
classical study of the dimension growth of countable groups. It is an open
question whether every finitely generated amenable group has subexponential
dimension growth, and (1) would be a strengthening of this. Indeed, it is
already an open problem whether the amenable group $\Z \wr (\Z \wr \Z)$ has
subexponential dimension growth \cite[Section 4]{DSX}\footnote{
This problem seems highly
related to the open problem in Dranishnikov-Sapir \cite[Question 3.9]{DSX}
of understanding the function $k \mapsto \sdim^{\Z^k}(2)$ of the dimension
growth of $\Z^k$ at scale 2 with the $\ell^1$ metric.}.
However we
note that the conjecture 
that a group has property A if and only if it has
subexponential dimension growth would imply that this is true, since every
amenable group has property A.
A key difficulty in this investigation is the open question of how wreath products affect dimension growth (see \cite[Section 4]{DSX}).
We note here that in parallel to this, wreath products have been
a key difficulty in further progress on Weiss's question; whether there is a
positive solution to Weiss's question for the group $\Z \wr \Z$ has been the key
next ``test problem'' after the work in \cite{CJMSTD}.

%In \Cref{subsec:FolnerA-conjecture} we also state a stronger version of part (1)
%of this conjecture for all Borel metric spaces of uniformly finite upper volume
%growth with Borel F{\o}lner property A.

Part (2) of Conjecture F would require a significant new advance
beyond the methods in \Cref{sec:hyperfiniteness}. Those methods originate in the
work of Gao and Jackson \cite{GJ} and have been a central tool for
hyperfiniteness proofs since they were introduced. 
But we have pushed them as far as possible here; in \Cref{rem:ab-leq-1/4} we
show rigorously that the condition $ab\leq 1/4$ in Theorem C cannot be improved further using
the methods of \Cref{sec:hyperfiniteness}.
However, we
remark that even modest progress here would already be very interesting. For
example, even showing that part (2) is true for Borel extended metric spaces of
sublinear dimension growth would substantially increase the exponent in our main
Theorem A. Further, since the dimension growth of spaces whose volume growth is
$\exp(O(r^\alpha))$ for some fixed $\alpha<1$ is polynomially bounded with
polynomial control by our Theorem D, proving this polynomial special case of
Conjecture F would imply that all Borel extended metric spaces whose volume growth is
$\exp(O(r^\alpha))$ for some fixed $\alpha<1$ are hyperfinite.

\subsection{Acknowledgements}

The authors would like to thank Anton Bernshteyn, Atticus Cull, Felix Weilacher, and Jing Yu
for helpful discussions about the material of the paper.

GPT 5.6 Pro was used to proofread a draft of this paper.

\section{Preliminaries}
\label{sec:prelim}

\subsection{Extended metric spaces}

An extended metric space is one where the metric may take the value $\infty$.
Formally, an \define{extended metric} on a set $X$ is a function $\rho : X^2 \to [0,
\infty]$ that is symmetric, satisfies the triangle inequality,
and where $x = x' \leftrightarrow \rho(x, x') = 0$ for all $x, x' \in X$. An \define{extended metric space}
$(X,\rho)$ is a set $X$ equipped with an extended metric. One example
of an extended metric space is the shortest path metric on a graph that has more than
one connected component, so points in different connected components are at
infinite distance. We will typically use the uppercase letters $X,Y,Z$ to denote 
metric spaces, omitting the metric $\rho$ when it is understood. We will typically use $x,
y, z$ for points in these spaces, and $A, B, C$ for subsets of them. If
$(X,\rho)$ is an extended metric space and $X' \subset X$ is a subset of $X$, we
can equip $X'$ with the restriction $\rho \restriction X'$ of the metric $\rho$
to $X'$ to obtain the sub-metric space $(X',\rho \restriction X')$ of $(X,\rho)$.

For the remainder of this section, assume $(X,\rho)$ is an extended metric space.
We define the distance from $x \in X$ to a set $A \subset X$ by 
$\rho(x,A) \coloneqq \inf_{y \in A} \rho(x,y)$, and 
if $A,B \subset X$ 
their distance is
$\rho(A,B) \coloneqq
\inf_{x \in A, y \in B} \rho(x,y)$. 
The \define{diameter} of $A$ is $\diam(A) \coloneqq \sup_{x,y \in A} \rho(x,y)$.
If $r \in [0,\infty)$ and $x \in X$, we let 
\[B_r(x) \coloneqq \{x' \in X : \rho(x, x') \le r\}\] 
denote the \define{closed $r$-ball around $x$}. Note that the cardinality
of any $0$-ball is $1$ since $B_0(x) = \{x\}$. If $A \subset X$ is a set,
we define $B_r(A) \coloneqq \{x' \in X : \rho(x',A) \leq r\}$ to be the set of all
points of distance at most $r$ from $A$. 
In case we need to specify the metric or space in our definitions (such as when we have more than one metric
on a space, or a point or set is contained in multiple different metric spaces), we
add $\rho$ or $X$ as a superscript to our notations such as
$B_r^X(A)$ or $\diam^{\rho}(A)$ to emphasize the metric space we are working
in.

We say an extended metric space $(X,\rho)$ is \define{locally finite} if
$B_r(x)$ is finite for every $x \in X$ and $r \in [0,\infty)$, and we say it is
\define{locally countable} if $B_r(x)$ is countable for every $x \in X$ and $r
\in [0,\infty)$. 
We define the \define{(uniform upper) volume growth} of the space $X$ to be
the function giving the supremum of the cardinality of all $r$-balls in
$X$ for each $r$:
\[\vol(r) \coloneqq \sup_{x \in X} |B_r(x)|.\]
We say that $X$ has uniformly bounded volume growth if $\vol(r) < \infty$
for all $r > 0$.

Some important classes of growth are as follows.
We say that a real valued function $f$ has \define{subexponential} growth if 
for every $b > 1$, $f(x) < b^x$ for all sufficiently large $x$. Equivalently,
$\log(f(x)) \in o(x)$.
One important growth bound that implies subexponential growth is the existence of
$C > 0$ and $\alpha < 1$ such that $f(x) \leq \exp(C x^{\alpha})$ for all
sufficiently large $x$. We will write $f \in \exp(O(x^\alpha))$ to mean this.
This condition is equivalent to the weaker condition that there is
$C, D > 0$ such that $f(x) \leq D \exp(C x^{\alpha})$ for all sufficiently large $x$.
A final condition on growth that is weaker than having subexponential growth
is non-exponential growth, i.e. not having an exponential function as a lower
bound.
We say that $f$
has \define{non-exponential} growth if there does not exist any $b > 1$ such that $f(x)
\geq b^x$ for all sufficiently large $x$.

If $A \subset X$, we define the \define{external $r$-boundary} of $A$ to be:
\[\bndry_r(A) \coloneqq  B_r(A) \setminus A.\] 
We say that a nonempty finite
set $A$ is \define{$(r,\ve)$-F{\o}lner} if $|\bndry_r(A)| \leq \ve
|A|$, or equivalently if $|B_r(A)| \leq (1 + \ve)|A|$. We say that $A$ is
simply \define{$\ve$-F{\o}lner} if it is $(1,\ve)$-F{\o}lner.

We will often deal with
families $\calU \subset \powset(X)$ of subsets of $X$, where $\powset(X)$
denotes the powerset of $X$, and we will use
uppercase calligraphic letters like $\calA, \calB, \calU$ to denote such
collections. All such families of sets we consider in this paper will be
collections of sets that all have finite diameter. 
The \define{support of $\calU$}, denoted
\[\supp(\calU) \coloneqq \bigcup_{A \in \calU} A\]
is the union of all its elements. $\calU \subset \powset(X)$ is said to be a \define{cover} of $X$ if
$\supp(\calU) = X$. We say that $\calU$ is a \define{disjoint family} if all distinct
$A, B \in \calU$ are disjoint. 
We say that
$\calU$ has \define{diameter uniformly bounded by $r$} if %$\mesh(\calU) \leq r$,
$\diam(A) \leq r$ for all
$A \in \calU$. 
We say that $\calU$ is \define{$s$-separated} if $\rho(A,B)
> s$ for all distinct $A, B \in \calU$. Hence, a $0$-separated family is
disjoint. 
We let 
\[B_r(\calU) \coloneqq \{B_r(A) : A \in \calU\}\]
denote the collection consisting of the closed $r$-balls around every $A \in \calU$.

Suppose $\calU = (\calU_0, \calU_1, \ldots, \calU_n)$ is a finite sequence of families of
subsets of $X$, where $\calU_i \subset \powset(X)$ for every $i$. 
We define the support of this
sequence 
\[\supp(\calU) = \supp(\calU_0, \ldots, \calU_n) \coloneqq \supp(\calU_0) \union \ldots
\union \supp(\calU_n)\] 
to be the union of the supports of the $\calU_i$. 
We will say that $\calU_0, \calU_1, \ldots, \calU_n$ \define{covers}
$X$ 
if $\supp(\calU_0, \ldots, \calU_n) = X$.

\subsection{Descriptive set theory}

Our standard reference for basic descriptive set theory is \cite{K95}. Our
descriptive set-theoretic definitions and conventions follow that book. 
A \define{Borel extended metric space} (see \cite[Section 2.3]{CJMSTD}) is a standard Borel space $X$
equipped with an extended metric $\rho \colon X^2 \to [0,\infty]$ such that
the function $\rho$ is Borel.

The \define{finite distance equivalence relation} on $(X,\rho)$, denoted
$E_\rho$, is the equivalence relation on $X$ where $x
\mathrel{E_\rho} x'$ if and only if $\rho(x, x') < \infty$. Note that if
$(X,\rho)$ is a locally countable extended metric space, then $E_\rho$ is a
\textbf{countable Borel equivalence relation} or CBER, meaning the relation
$E_{\rho}$ is Borel as a subset of $X \times X$, and all of
its equivalence classes are countable. A CBER $E$ is \define{hyperfinite} if it
is an increasing union $E = \bigunion_n F_n$ of CBERs $F_0 \subset F_1 \subset \ldots$ with finite classes. See \cite{K25} for a recent survey
of the theory of countable Borel equivalence relations, which is our
standard reference for concepts of hyperfiniteness of
CBERs, and basic properties about them. 

Most of the families of sets $\calF \subset \powset(X)$ that we will use in
the paper will be disjoint families of sets. We will call such a family
\define{Borel} if $\supp(\calF)$ is Borel, and the equivalence relation
$E_{\calF}$ on $\supp(\calF)$ whose classes are exactly the sets in $\calF$
is a Borel equivalence relation on $\supp(\calF)$. 

In this paper, we largely work in the setting where the metric space
$(X,\rho)$ is locally finite, and our families $\calF \subset \powset(X)$
have bounded diameter, and thus each $A \in \calF$ is finite. 
In this case, we could equivalently work
with the standard Borel space $[X]^{< \infty}$ of finite subsets of $X$, in
which case $\calF$ is Borel as a subset of that space iff it is Borel in
the sense of the above definition (i.e. if $\supp(\calF)$ and $E_\calF$ are
Borel).

We omit some routine verifications that
some of our constructions are Borel in our proofs, since they are defined by
``local rules'' and are hence Borel by \cite[Lemma 5.17]{P}, or where the
Borelness follows from Lusin-Novikov uniformization. 

\subsection{Dimension growth}
\label{subsec:dim-growth}

Recall asymptotic dimension was introduced by Gromov \cite{G93}, and one of
its equivalent definitions uses the following notion we will call an
$(s,d)$-cover, following \cite{JYX}: 
\begin{defn}[\cite{JYX}]
  Suppose $(X,\rho)$ is an extended metric space and $s,d \in [0,\infty)$.
  An \define{$(s,d)$-cover}
  $\calU =(\calU_0, \ldots, \calU_n)$ of $X$ is a cover where each $\calU_i$
  is $s$-separated and has diameter uniformly bounded by $d$. We say
  that the $(s,d)$-cover $\calU=(\calU_0, \ldots, \calU_n)$ has $n+1$
  \define{elements} or \define{colors}.
\end{defn}

Equivalently, one may color the points of $X$ by $n+1$ colors and consider
the equivalence relation generated by pairs of points of the same color
that are at distance at most $s$, which is called the $s$-walk equivalence
relation. Then the existence of an $(s,d)$-cover
is equivalent to the statement that the $s$-walk equivalence relation has classes of diameter
at most $d$.

Using this definition of an $(s,d)$-cover, $X$ has \define{asymptotic dimension} $n \in \N$ if $n$ is
the least such that for every $s \geq 0$ there exists $d$ for which
there is an $(s,d)$-cover of $X$ with $n + 1$ elements. 
(See \cite{BD} for a discussion of many other equivalent definitions).

More generally, it is natural to consider the two-parameter function that records,
for each separation bound $s$ and diameter bound $d$, the least $n$ such that
there is an $(s,d)$-cover of $X$ with $n+1$ elements.
This was first formally defined in the paper of 
Dranishnikov-Sapir \cite{DSX}.\footnote{We emphasize that these definitions come from the
expanded and updated 2012 arXiv v4 of the paper and not the 2011 published
version of the paper \cite{DS11}. Some errors in that version were corrected
in a corrigendum \cite{DS12}, prompting the expanded arXiv v4: the ``final
version'' of the paper.}
Precisely, the \define{$(s,d)$-dimension} of a metric space $X$ is: 
\[
  \dim^X(s,d) \coloneqq \inf \{n \colon \text{there is an $(s,d)$-cover of $X$ with $n+1$ elements}\}. \]
Note that $\dim^X(s,d)$ is increasing in the first coordinate, and
decreasing in the second coordinate.
The limit of $\dim^X$ gives the usual asymptotic dimension of $X$:
\[
  \asdim(X) \coloneqq \lim_{s \to \infty}
  \lim_{d \to \infty} \dim^X(s,d).
\]

The \define{single-variable dimension growth function}
\[\sdim^X(s) = \lim_{d \to \infty} \dim^X(s,d)\]
has also been studied extensively, and was also originally defined by
Gromov~\cite[Section 6.F]{G99}, who called it the dimension of $X$ on the scale
$s$.
It is an open question whether a finitely
generated group has property A if and only if it has subexponential dimension growth
(see \cite[Question 2.13]{DSX} and \cite[Page 922]{O14}, and see also
\cite{O12}).
We mention also that if $d \colon [0,\infty) \to [0,\infty)$ is
an increasing function, then \define{the $d$-controlled dimension of $X$} is the
function $s \mapsto \dim^X(s,d(s))$. For example, \cite{DSX} proves many
bounds on the $d$-controlled dimension growth for various groups and choices of functions
$d$.

We remark that an alternate definition of asymptotic dimension based on
``uncolored covers'' is as follows. Say the $r$-multiplicity of a cover $\calU \subset \powset(X)$
of $X$ is $\sup_{x \in X} |\{U \in \calU \colon U \inters B_r(x) \neq \emptyset\}|$,
that is 
the supremum of the number of elements of $\calU$ that are met by some $r$-ball.
Then $\asdim(X) \leq n$ if and only if for
every $r$ there exists $d$ and a cover $\calU \subset \powset(X)$ of $X$ by sets of
diameter at most $d$ so that the $r$-multiplicity of $\calU$ is $\leq n + 1$
\cite[Theorem 19]{BD}. Based on this equivalent definition of asymptotic
dimension, one could define ``uncolored'' dimension growth
functions. For example, let $\udim^X(r,d) = \inf \{n \colon $ there exists a
cover $\calU \subset \powset(X)$ of $X$ by sets of diameter at most $d$ so that
the $r$-multiplicity of $\calU$ is $n + 1\}$. We note that $\udim^X(s/2,d) \leq
\dim^X(s,d)$ simply by taking the union of the elements of a ``colored'' cover to obtain an uncolored
one.
Dranishnikov has defined yet
another related dimension growth function based on Lebesgue numbers \cite{D06}.
The exact relationship between all these colored vs uncolored
dimension growth functions is an open problem
\cite[Question 1.5]{DSX}. We will not study these ``uncolored'' dimension
growth functions; they do not correspond to the original definition of
dimension growth given by Gromov, and they are less related to our main
concerns.

This paper studies a Borel version of 
Dranishnikov-Sapir's two-parameter dimension growth function, which we
define for locally finite metric spaces. 

\begin{defn} \label{defn:dim-growth}
  Let $(X,\rho)$ be a locally finite Borel extended metric space. 
  A \define{Borel $(s,d)$-cover} of $X$ is an $(s,d)$-cover
  $\calU = (\calU_0, \ldots, \calU_n)$ where each $\calU_i$ is Borel.
  We define the two-parameter Borel dimension function of $X$ by:    
  \[
    \dim^X_B(s, d) \coloneqq \inf\{n \in \N \colon
      \text{there is a Borel $(s,d)$-cover of $X$ with $n+1$
      elements}\}.
  \]
  We define the associated single-variable Borel dimension function of $X$ by:
  \[\sdim^X_B(s) \coloneqq \lim_{d \to \infty} \dim_B^X(s,d).\]
\end{defn}

Note that the Borel asymptotic dimension of $X$ as defined in
\cite{CJMSTD} for locally finite metric spaces is related to its Borel dimension growth by
\[
  \asdim_B(X) = \lim_{s \to \infty} \lim_{d \to \infty}
  \dim_B^X(s,d) = \lim_{s\to \infty} \sdim_B^X(s).
\]

We note that one key upper bound on $\dim_B^X(s,d)$ is $\vol(s)$, just as it is
classically (see \cite[Lemma 2.3]{DSX}).

\begin{prop}
  Suppose $(X,\rho)$ is a Borel extended metric space of uniformly bounded volume growth.
  Then
  $\dim_B^X(s,d) \leq \vol(s)-1$ for all $d$, and so $\sdim_B^X(s) < \vol(s)$. 
\end{prop}
\begin{proof}
  Let $G$ be the graph on $X$ where distinct
  $x$ and $y$ are adjacent if $\rho(x,y) \leq s$. So each vertex has at most
  $\vol(s) - 1$ neighbors. 
  Then $G$ has a Borel $\vol(s)$-coloring $f \colon X \to \{0, \ldots,
  \vol(s)-1\}$ by \cite[Proposition
  4.6]{KST99}. So letting $\calU_i = \{\{x\} \colon f(x) = i\}$ be the singletons
  consisting of points of color $i$, each family $\calU_i$ is Borel, $s$-separated, and has
  diameter $0$, and hence $\calU_0, \ldots, \calU_{\vol(s) -1}$ is an
  $(s,0)$-cover of $X$.
\end{proof}

Hence, two particularly important classes of Borel extended metric spaces
always have $\dim^X_B(s,d)$ and hence $\sdim^X_B(s)$ bounded above by
exponential functions in $s$, since they have exponential volume growth: the path
length metric on bounded degree Borel graphs, and 
the word metric on a Borel action of a finitely generated group. 
Hence, a key dividing line here will be between those spaces with
exponential dimension growth and non-exponential dimension growth.

In \Cref{sec:mu-hyperfiniteness} we will also consider the Borel separation
index of $X$, by relaxing the condition that the collections $\calU_i$ of sets
in the covers of $X$ have uniformly
bounded diameter to the requirement that these sets each have bounded diameter:
\begin{defn}
The Borel separation index of a locally finite Borel extended metric space $(X,\rho)$
is the function defined by 
$\si^X_B(s) \coloneqq \inf \{n \colon$ there is a Borel cover $\calU_0,
\ldots, \calU_n$ of $X$ by $n+1$ families of sets of bounded diameter such that each
$\calU_i$ is $s$-separated$\}$. If $\mu$ is a finite Borel measure on $X$, we also 
define the $\mu$-measurable separation index function: 
$\si^X_\mu(s) \coloneqq \min \{\si^{X'}_B(s) \colon X' \subset X \text{is an
$E_\rho$-invariant $\mu$-conull Borel set}\}$. Finally, 
define the $\mu$-measurable asymptotic separation index by
$\asi_\mu(X) \coloneqq \lim_{s \to \infty} \si^X_\mu(s)$.
\end{defn}

Note the Borel asymptotic separation index $\asi_B(X)$ of $X$ as defined in
\cite{CJMSTD} is the limit of its Borel separation index function:
\[\asi_B(X) = \lim_{s \to \infty} \si^X_B(s).\]

Note that $\si^X_B(s) \leq \sdim_B^X(s)$, though the inequality may be
strict since a sequence $\calU_0, \ldots, \calU_n$ witnessing $\si^X_B(s) = n$
may consist of finite sets that are not of uniformly bounded diameter. 
For example, by the results of \cite{CJMSTD}, a free Borel action of $\Z^k$ on a
standard Borel space $X$ with the usual word metric gives an extended metric
space $X$ with $\asdim_B(X) = k$, but
$\asi_B(X) = 1$. So if $k \geq 2$, then $\si^X_B(s) = 1 < k = \sdim_B^X(s)$
for all sufficiently large $s$.
% Indeed, there are
% hyperfinite Borel extended metric spaces of uniformly bounded volume growth for which $\asi_B(X)=1$ but $\sdim_B^X$ has exponential growth can be constructed
% by taking a hyperfinite Borel
% graph with $\asi_B(X) = 1$ on $X$, and then taking a countable collection of
% disjoint complete sections $(A_n)_{n \in \N}$ for the equivalence relation
% $E_\rho$, and then replacing each element
% of $A_n$ 
% with ``finite gadgets'' that are expander
% graphs of size $n$, as in \cite[Corollary 2.10]{DSX}. 
% We leave the details as an exercise for the reader. 

We finish by briefly noting how rescaling a metric changes its volume growth and dimension
function.
\begin{prop}\label{prop:rescaling}
  Suppose $(X,\rho)$ is an extended metric space and $c > 0$. If $c \rho$
  is the rescaling of the metric $\rho$ by $c$, then for all
  $r,s,d > 0$,
  \[
    \vol^{c\rho}(cr) = \vol^\rho(r)
    \quad\text{and}\quad
    \dim^{\rho}(s,d) = \dim^{c\rho}(cs,cd).
  \]
  Hence, if $(X,\rho)$ is a Borel locally finite extended metric
  space, then
  \[
    \dim_B^{\rho}(s,d) = \dim_B^{c\rho}(cs,cd).
  \]
\end{prop}
\begin{proof}
This follows since $B_{cr}^{c \rho}(x) = B_r^\rho(x)$, and a set has $c\rho$-diameter $cd$
iff it has $\rho$-diameter $d$ and a family of sets is $cs$-separated for $c
\rho$ iff it is $s$-separated for $\rho$.
\end{proof}

We note one consequence of rescaling: if $\alpha > 0$, $C > 0$, $b > 1$, and $\vol^{\rho}(r) \leq C
b^{r^\alpha}$ for all $r \in [0,\infty)$, then $\vol^{c \rho}(r) \leq C
b^{(r/c)^\alpha} = C \big(b^{c^{-\alpha}}\big)^{r^\alpha}$. Hence, by
rescaling the metric, we can change the base in this exponential
volume growth bound to be any real number greater than $1$. Hence, the base
is unimportant in our results, and we typically use the base of the natural
logarithm $e$ in the statement of our theorems.

\section{F{\o}lner layering and Ball carving}
\label{sec:ball-carving}

The key idea we use to analyze the dimension growth of spaces of subexponential growth will be a certain way of decomposing such
a space inductively using finitely many maximal $r$-separated disjoint
collections of $(r,\ve)$-F{\o}lner sets. These ideas originate from
ball-carving algorithms from distributed computing in theoretical computer
science (see \cite[Theorem 1.10]{R}). The following definition will be key
to this analysis:

\begin{defn}
\label{defn:Folner-layering}
  Let $r,\ve > 0$.
  An \define{$(r,\ve)$-F{\o}lner layering} of an extended metric space $X$
  is a finite sequence $\calF_0, \ldots, \calF_{n-1}$ of families with associated subspaces 
  $X_0 = X$ and $X_{i+1} \coloneqq X_i \setminus B_r(\supp(\calF_i))$
  for $i<n$, so that each $\calF_i$ has uniformly bounded diameter, and so that for all $i<n$, 
  \begin{enumerate}

  \item $\calF_i \subset \powset(X_i)$ is a family of $(r,\ve)$-F{\o}lner
  subsets of $X_i$.

  \item $\calF_i$ is $r$-separated, 

  \item $B_r(\supp(\calF_0,\ldots,\calF_{n-1})) = X$.

  \end{enumerate}
  A \define{$\ve$-F{\o}lner layering} of $X$ is an $(1,\ve)$-F{\o}lner layering. 
\end{defn}

In item (1) in \Cref{defn:Folner-layering} we emphasize that we mean a F{\o}lner subset of the
subspace $X_i$ and not a F{\o}lner subset of $X$. That is,
for all $A \in \calF_i$, $A \subset X_i$, and $|B^{X_i}_r(A)| \leq (1 + \ve)
|A|$.

Note that the
union $\calF_0 \cup \ldots \cup \calF_{n-1}$ of all the sets in an
$(s,\ve)$-F{\o}lner layering is
$s$-separated, since each $\calF_i$ is $s$-separated and $\calF_j$ is
disjoint from $B_s(\supp(\calF_i))$ for $j > i$ by the definition of $X_j$
and the condition that $\calF_j \subset \powset(X_j)$.
We will eventually use these $s$-separated families coming from F{\o}lner
layerings to make $(s,d)$-covers which we use to bound the dimension growth
of metric spaces of subexponential growth.

Our first goal will be to show that every Borel extended metric space of
subexponential volume growth has an $(r,\ve)$-F{\o}lner layering for every
$r,\ve > 0$.
It will suffice to prove theorems just in the case
$r = 1$, and then by rescaling the metric we can obtain similar results for
all $r$. Hence most of our results below will just be
about $\ve$-F{\o}lner layerings.

We begin with a well known proposition showing some ball must be an
$\ve$-F{\o}lner set in any space of uniformly subexponential growth. 

\begin{prop}\label{prop:subexp-Folner-balls}
  Suppose $(X,\rho)$ is an extended metric space, $\ve > 0$, and $r$ is a positive integer
  such that $\vol^X(r+1) \leq (1 + \ve)^{r+1}$. Then for every
  $x \in X$, there is an
  $\ve$-F{\o}lner set $F$ containing $x$ such that $F \subset B_r(x)$.
\end{prop}
\begin{proof}
  \[
    \frac{|B_1(x)|}{|B_0(x)|}
    \frac{|B_2(x)|}{|B_1(x)|}
    \cdots
    \frac{|B_{r+1}(x)|}{|B_r(x)|}
    = \frac{|B_{r+1}(x)|}{|B_0(x)|}
    \leq \vol(r+1)
    \le (1 + \ve)^{r+1},
  \]
  so there is some $s \le r$ such that
  $\frac{|B_{s+1}(x)|}{|B_s(x)|} \le 1 + \ve$. So $F = B_s(x)$ is the desired
  $\ve$-F{\o}lner set.
\end{proof}

Note that the condition $\vol^X(r+1) \leq (1 + \ve)^{r+1}$ will hold for
sufficiently large $r$ if $X$ has subexponential volume growth.
Now we can show that a Borel $\ve$-F{\o}lner layering exists in a space
with this volume growth bound:
\begin{lem}\label{layering-existence}
  Suppose $\ve > 0$, $r$ is a fixed positive integer, and $X$ is a locally finite Borel extended metric
  space such that $\vol^X(r+1) \le (1 + \ve)^{r+1}$.
  Then there is a Borel $\ve$-F{\o}lner layering
  of $X$ by sets of diameter uniformly bounded by
  $2r$.
\end{lem}
\begin{proof}
  Set $X_0 = X$. Given $X_i$, let $\calF_i$ be
  a Borel maximal $1$-separated family of
  $\ve$-F{\o}lner sets in $X_i$ of diameter bounded by $2r$, and then let
  $X_{i+1} = X_i \setminus B_1(\supp(\calF_i))$.
 To see that such a $\calF_i$ exists, consider the Borel graph on 
 all $2r$-bounded $\ve$-F{\o}lner subsets of $X_i$ where distinct $A, B$ are
 adjacent if $\rho(A,B)\leq 1$. This graph is locally finite since $\rho$ is
 locally finite. Hence it has a Borel maximal independent set by
 \cite[Prop 4.2, 4.5]{KST99}.

  We claim that after $n = \vol(r)$ steps of this process, the $1$-balls around the sets
  $\calF_0, \ldots, \calF_{n-1}$ cover the space: $B_1(\supp(\calF_0, \ldots,\calF_{n-1})) =
  X$, and hence $\calF_0, \ldots, \calF_{n-1}$ is the desired $\ve$-F{\o}lner
  layering.
  Since $X_i$ is a subspace of $X$,
  its volume growth is bounded by 
  $\vol^{X_i}(r+1) \le (1 + \ve)^{r+1}$. Hence, by
  \Cref{prop:subexp-Folner-balls} applied to the space $X_i$, every
  $x \in X_i$ is contained in some $\ve$-F{\o}lner set $A \subset
  B^{X_i}_r(x)$. Maximality of $\calF_i$ implies that $A$ meets $B_1^{X_i}(F)$
  for some $F \in \calF_i$. Hence, since $X_{i+1} = X_i \setminus
  B_1(\supp(\calF_i))$, we have shown
  \[
  |B_{r}(x) \cap X_{i+1}| < |B_{r}(x) \inters X_i|
  \]
  for every $x \in X_i$. This implies $X_n$ is empty since if $x \in X_n$, then 
  $|B_r(x) \inters X_i|$, for $0 \leq i \leq n$, would be a decreasing sequence of
  positive integers, but $|B_r(x) \inters X_0| = |B_r(x)| \leq \vol(r) = n$. 
\end{proof}

As a corollary, every Borel extended metric space of uniformly bounded
subexponential volume growth has F{\o}lner layerings, by a rescaling argument. We will use this fact in
\Cref{sec:folner-tiling}.

\begin{cor}\label{cor:subexp-layering}
Suppose $(X,\rho)$ is a Borel extended metric space with uniformly bounded
subexponential volume growth. Then for every $r,\ve > 0$, $X$ has a Borel
$(r,\ve)$-F{\o}lner layering.
\end{cor}
\begin{proof}
Let $\rho' = r^{-1}\rho$. This rescaled metric still has uniformly bounded
subexponential volume growth (see \Cref{prop:rescaling}). 
Hence, for some positive integer $d$,
$\vol^{\rho'}(d+1) \leq (1+\ve)^{d+1}$, and so 
by \Cref{layering-existence}, $(X,\rho')$ has a Borel $(1,\ve)$-F{\o}lner
layering with respect to $\rho'$. 
Since a family is $1$-separated with respect to $\rho'$ if and only if it is $r$-separated
with respect to $\rho$, and uniformly bounded diameters remain uniformly
bounded under rescaling, this layering is a Borel
$(r,\ve)$-F{\o}lner layering with respect to $\rho$.
\end{proof}

If we remove the sets in a F{\o}lner layering from a space $X$, we can bound the
volume growth of the subspace that remains in terms of the volume growth of $X$ as
follows:
\begin{lem}\label{layering-complement}
  Suppose $X$ is a locally finite extended metric space, and 
  $\calF_0, \ldots, \calF_{n-1}$ is an $\ve$-F{\o}lner
  layering of $X$ of diameter bounded by $d \geq 0$. Let $Y =
  X \setminus \supp(\calF_0, \ldots, \calF_{n-1})$. 
  Then for every finite set $A \subset Y$, we have $|A| \le \ve |B^X_{d+1}(A)|$.
  Hence, 
  \begin{equation}\label{eqn:volume-slowdown}
  \vol^{Y}(r) \leq \ve \vol^{X}(r + d +1)
  \end{equation} 
  for all $r \geq 0$.
\end{lem}
\begin{proof}
  Since $A \subset Y = X \setminus \supp(\calF_0, \ldots, \calF_{n-1})$ and $X =
  B_1(\supp(\calF_0, \ldots, \calF_{n-1}))$ by the definition of a F{\o}lner layering,
  every $y \in A$ is contained in $\bndry^{X_i}_1(C) = B^{X_i}_1(C)
  \setminus C$ for some $i<n$ and $C \in \calF_i$.
  Let $\calC_i = \{C \in
  \calF_i \colon \bndry^{X_i}_1(C) \inters A \neq \emptyset\}$ be the
  elements of $\calF_i$ whose $1$-boundary intersects $A$. So $A \subset
  \supp(\bndry_1^{X_0}(\calC_0), \ldots, \bndry_1^{X_{n-1}}(\calC_{n-1}))$.
  Since the
  sets in $\calF_i$ have diameter bounded by $d$, 
  $\supp(\calC_0, \ldots, \calC_{n-1}) \subset B^X_{d+1}(A)$. Note that
  $|\supp(\calC_0, \ldots, \calC_{n-1})| = |\supp(\calC_0)| + \ldots +
  |\supp(\calC_{n-1})|$ since the sets $\supp(\calC_i)$ are pairwise disjoint.
  Finally, note that since every $C \in \calC_i$ is an
  $\ve$-F{\o}lner set in $X_i$, and a finite union of disjoint $\ve$-F{\o}lner sets is
  $\ve$-F{\o}lner, $\supp(\calC_i)$ is an $\ve$-F{\o}lner subset of $X_i$. Hence,
  \begin{multline*}
  |A| \leq |\supp (\bndry_1^{X_0}(\calC_0))| + \ldots +
  |\supp(\bndry_1^{X_{n-1}}(\calC_{n-1}))| \leq \\
  \ve
  |\supp(\calC_0)| + \ldots + \ve |\supp(\calC_{n-1})| =
\ve
  |\supp(\calC_0,\ldots,\calC_{n-1})| \leq \ve |B^X_{d+1}(A)|.
  \end{multline*}

  Finally, \Cref{eqn:volume-slowdown} follows by applying
  the above to the sets $A = B^Y_r(x)$ for each
  $x$, since then $B^X_{d+1}(A) \subset B^X_{r + d + 1}(x)$, so $|B^Y_{r}(x)|
  \leq \ve |B^X_{r + d + 1}(x)|$.
\end{proof}  

We will use the above lemma to show that in a metric space of subexponential
growth, if we repeatedly remove F{\o}lner layerings from the space, 
after finitely many steps, the F{\o}lner layerings must cover
the entire space. Hence this gives an upper bound on the Borel dimension
$\dim_B^X(s,d)$ of Borel extended metric spaces of subexponential growth. 

\begin{lem}\label{prop:dim-estimate}
  Suppose $\ve > 0$ and $r, n$ are positive integers. 
  Then every locally finite Borel extended metric space $X$
  satisfying $\vol^X(r+1) \le (1 + \ve)^{r+1}$
  and $\vol^X((2r+1)n) < \frac{1}{\ve^n}$
  has $\dim_B^X(1,2r) < n$.
\end{lem}
\begin{proof}
  Let $Y_0 = X$. For each $i \geq 0$, by \Cref{layering-existence},
  let $\calF_i = (\calF_{i,0}, \ldots, \calF_{i,n_i-1})$ be a Borel $\ve$-F{\o}lner layering of $Y_{i}$ of diameter bounded
  by $2r$, and let $\calU_i = \calF_{i,0} \union \ldots \union \calF_{i,n_i-1}$ be
  the union of all the sets in this layering, so $\calU_i$ is $1$-separated and has
  diameter bounded by $2r$.
  Let $Y_{i+1}
  = Y_{i} \setminus \supp(\calU_i)$.  
  It suffices to show that $Y_n$ is empty,
  since then $\calU_0, \ldots, \calU_{n-1}$ is a Borel $(1,2r)$-cover of $X$.

  After deleting each layering, note that 
 $\vol^{Y_{i+1}}(t)\leq \ve\vol^{Y_i}(t+2r+1)$
  for every $t\geq0$, by applying 
  \Cref{layering-complement} to $Y_i$. Hence, by induction
 $\vol^{Y_n}(0) \leq \ve^n \vol^{Y_0}((2r+1)n) < 1$, and so $Y_n$ is empty.
\end{proof}

% Note that if $X$ is a space of subexponential growth, then for any $\ve > 0$,
% the two conditions $\vol(r+1) \le (1 + \ve)^{r+1}$ and $\vol((2r+1)n) <
% \frac{1}{\ve^c}$ will be satisfied for sufficiently large $r$, and any
% sufficiently large $n$ given $r$.

\begin{remark}
We remark that \Cref{prop:subexp-Folner-balls}, \Cref{layering-existence},
\Cref{layering-complement}, and
\Cref{prop:dim-estimate} together encapsulate the simple deterministic ball
carving algorithm from the theory of local algorithms originally introduced in
\cite{AGLP89}, see also \cite[Section 1.5]{R}. In that
language, a $(n,d)$-network decomposition of a graph is a cover of the
vertices by families $\calC_{0}, \ldots, \calC_{n-1}$ (called clusters) of
subsets where each
$\calC_i$ is $1$-separated, and has diameter uniformly bounded by $d$. In our language, this
is a $(1,d)$-cover with $n$ elements. 
The ball carving algorithm makes each cluster $\calC_i$ by 
sequentially starting at each point $x$ of the space and taking the least radius
$s$ for which the
ball $B_s(x)$ is $\ve$-F{\o}lner, as in
\Cref{prop:subexp-Folner-balls}. Then one
adds this set to the cluster and removes both it and the $1$-ball
around it from the space to make a collection of $1$-separated sets. This
cluster
corresponds to what we 
are calling a F{\o}lner layering, where each layer in our F{\o}lner layering
roughly corresponds to a single step in this clustering process which takes place at far enough
apart distance in the graph to not interfere. The ball carving algorithm then removes the points
in $\supp(\calC_i)$ from the
space, and does the same process again to make the next cluster $\calC_{i+1}$. Then the analysis in 
\Cref{layering-complement} and \Cref{prop:dim-estimate} shows a bound on the
total number of clusters made this way that are needed to cover the space.
\end{remark}

We record a corollary of the above, which comes from a careful choice of
$\ve$ in the above lemma.
% In particular, this gives a bound on the dimension of a Borel extended metric
% space from bounds on its growth:
\begin{cor}\label{cover-existence-discrete}
  If $X$ is a locally finite Borel extended metric space, satisfying
  \[
    \left(\vol(r+1)^{1/(r+1)} - 1 \right)^n \vol((2r+1)n) < 1,
  \]
  for $r, n \in \N$, then $\dim_B^X(1, 2r) < n$.
\end{cor}
\begin{proof}
  \trivial{Allowing zero as a possible element of $\N$ causes no problem.
  With the usual convention $t^0=1$, if $n=0$, then the left hand side of
  the displayed hypothesis is $\vol(0)=1$, so the hypothesis forces $n>0$.
  Set $\ve=\vol(r+1)^{1/(r+1)}-1$. If $\ve=0$, then
  $\vol(r+1)=1$, so the family of all singletons is a Borel $(1,0)$-cover;
  hence $\dim_B^X(1,2r)=0<n$. If $\ve>0$ and $r=0$, then
  $\ve=\vol(1)-1\geq1$, and the left hand side of the hypothesis is at
  least $1$, a contradiction. Thus in the remaining case $r,n$ are positive
  integers and $\ve>0$, exactly as required by
  \Cref{prop:dim-estimate}.}
  This follows by applying \Cref{prop:dim-estimate} with $\ve \coloneqq \vol(r+1)^{1/(r+1)} - 1$.
\end{proof}

We record another corollary which uses a rescaling argument to bound dimension
growth for all separations, and allows real number parameters for both the
separation and diameter. 
% Note that $r$ must be an integer in \Cref{cover-existence-discrete} above
% since its proof ultimately uses \Cref{prop:subexp-Folner-balls} where the
% argument expanded a radius by $1$, $r$ many times.
\begin{lem}\label{cover-existence-real}
  Suppose $X$ is a locally finite Borel extended metric space. Then for all $s > 0$,
  $d \geq 2s$, and positive integers $n$, if
  $\left(\vol(d)^{2s/d} - 1\right)^n \vol(2dn) < 1$
  then $\dim_B^X(s, d) < n$.
\end{lem}
\begin{proof}
  Define $r := \left\lfloor \frac{d}{2s} \right\rfloor$, and
  note that $r \geq 1$. So $s(r+1) \leq s(2r) \leq d$,
  and $r + 1 \geq \frac{d}{2s}$,
  hence
  \[
    \left(\vol(s(r+1))^{1/(r+1)} - 1\right)^n
    \le \left(\vol(d)^{2s/d} - 1\right)^n,
  \]
  since $\vol$ is nondecreasing. 
  We also have $s(2r+1)n \leq s(4r)n \leq 2dn$, and so
  \[
    \vol(s(2r+1)n) \le \vol(2dn).
  \]
  And hence 
  \[\left(\vol(s(r+1))^{1/(r+1)} - 1\right)^n \vol(s(2r+1)n) < 1.\]
  If we rescale the metric $\rho$ by the factor $s^{-1}$, then 
  $\vol^{s^{-1}\rho}(r) = \vol^{\rho}(sr)$ by 
  \Cref{prop:rescaling}. So by \Cref{cover-existence-discrete}
  \[\dim_B^{\rho}(s,d) = \dim_B^{s^{-1}\rho}(1,d/s) 
  \leq \dim_B^{s^{-1}\rho}(1,2r) < n.\]
\end{proof}

We note that this result already gives a simpler proof of the Bernshteyn-Yu
result \cite[Corollary 3.16]{BY} that Borel metric spaces of polynomial growth
$O(r^k)$ have Borel asymptotic dimension at most $k$, and hence their Theorem from
\cite{BY} that every Borel graph of polynomial growth is hyperfinite
\cite[Corollary 1.16]{BY}.

\begin{cor}\label{polynomial-growth-dimension}
  Suppose $X$ is a Borel extended metric space and $\vol^X(r) \in O(r^k)$ for
  $k \in \N$. Then 
  $\dim_B^X(s,s^{k+2})\leq k$ for all sufficiently large $s > 0$, hence
  $\asdim_B(X) \leq k$, since $\dim_B^X(s,d)$ is decreasing in $d$ and
  increasing in $s$. 
\end{cor}
\begin{proof}
  Let $s$ be sufficiently large, let $d \coloneq s^{k+2}$,
  and note that $d \geq 2s$ for sufficiently large $s$.
  We will apply \Cref{cover-existence-real} with $n = k+1$. 
  Since $\vol(r)\in O(r^k)$, 
  taking
  $k$ as a fixed constant, and computing growth as a function of $s$ we have:
  \[
    \frac{2s}{d}\log \left(\vol(d) \right) = \frac{2s}{d}\log \left(\vol(
    s^{k+2}) \right) \in O(s^{-k-1}\log s ) \to 0 \text{ as $s \to \infty$}.
  \]
  Hence
  $\vol(d)^{2s/d} =\exp(O(s^{-k-1}\log s))$, and since $\exp(x) \leq 3x+1$ for
  sufficiently small $x \in (0,1)$, we have $\vol(d)^{2s/d}-1 =O(s^{-k-1}\log
  s)$.

  Note also
    $\vol(2d(k+1))=\vol(O(s^{k+2})) \in O(s^{k(k+2)})$.
  Therefore
  \begin{multline*}
    \left(\vol(d)^{2s/d}-1\right)^{k+1}
      \vol(2d(k+1))
    = O\left( \left(s^{-k-1}\log s \right)^{k+1} s^{k(k+2)} \right) =
    O\left(
      s^{-(k+1)^2+k(k+2)}(\log s)^{k+1}
    \right)\\
    =O\left(\frac{(\log s)^{k+1}}{s}\right)
    \to 0
  \end{multline*}
  as $s \to \infty$.
  Thus, by \Cref{cover-existence-real},
    $\dim_B^X(s,d)<k+1$
  for all sufficiently large $s$ and hence
  $\dim_B^X(s,d)\leq k$.
\end{proof}

We can now bound the Borel dimension growth of any Borel extended metric space
of growth $\vol(r) \in \exp(O(r^{\alpha}))$ for some $\alpha \in
(0,1)$, which gives Theorem D from the introduction:

\begin{thm}\label{dimension-subexp}
  Suppose $X$ is a Borel extended metric space of volume growth bounded by 
  \[\vol(r) \in \exp(O(r^{\alpha}))\]
  for some $\alpha \in (0,1)$.
  Then there exists a constant $C > 0$ such that 
  \[\dim^X_B(s, Cs^{1/(1-\alpha)}) \in O(s^{\alpha/(1-\alpha)^2}).\]
\end{thm}
\begin{proof}
  Fix $b > 1$ such that $\vol(r) \leq b^{r^\alpha}$ for sufficiently large $r$.
  Fix $C > 0$ so that $\delta := b^{2/C^{1-\alpha}} - 1 < 1$,
  and fix $D > 0$ so that $\delta b^{2^\alpha C^\alpha/D^{1-\alpha}} < 1$.
  Define
  \[d(s) \coloneqq C s^{1/(1-\alpha)} \text{ and } c(s) \coloneqq \left\lceil D
  s^{\alpha/(1-\alpha)^2} \right\rceil + 1.\]

  For sufficiently large $s$, if we let $d = d(s)$ and $c = 
  c(s)-1$, then we have $d \geq 2s$, and
  \[
    \vol(d)^{2s/d}
    \le (b^{d^\alpha})^{2s/d}
    = b^{2s/d^{1-\alpha}}
    = b^{2/C^{1-\alpha}}
    = \delta + 1
  \]
  and
  \[
    \vol(2dc)^{1/c}
    \le (b^{(2dc)^\alpha})^{1/c}
    = b^{2^\alpha d^\alpha/c^{1-\alpha}}
    \leq b^{2^\alpha C^\alpha/D^{1-\alpha}}
    < \delta^{-1}.
  \]

  Hence
  \[
    (\vol(d)^{2s/d} - 1)^c \vol(2dc)
    \leq \delta^c (b^{2^\alpha C^\alpha/D^{1-\alpha}})^c
    < 1.
  \]
  So by \Cref{cover-existence-real}, $\dim^X_B(s,d(s)) < c+1 = c(s)$.
\end{proof}

\section{Hyperfiniteness}
\label{sec:hyperfiniteness}

In this section, we begin by proving Theorem C that shows sufficiently slow
Borel dimension growth implies hyperfiniteness. Our result generalizes
\cite[Theorem 7.3]{CJMSTD} which itself generalized the main construction in
Gao-Jackson \cite{GJ}. Our proof is a much more careful quantitative refinement
of the same basic idea from \cite[Theorem 7.3]{CJMSTD}.

We make a few definitions for the purposes of the proof:
\begin{defn}
Suppose $X$ is a set, and 
$\calU = (\calU_0, \ldots, \calU_n)$ is a finite sequence of collections
$\calU_i \subset \powset(X)$. If $A \subset X$, we define the \define{star of $A$ with
respect to $\calU$} to be the union of all the sets from the $\calU_i$ that it
intersects:
\[A * \calU = \bigcup \{B  \colon \exists i (B \in \calU_i) \land B \cap A \neq \emptyset\}.\]
If $\calA \subset \powset(X)$ is a family of subsets of $X$, we similarly define the \define{star of $\calA$ with
respect to $\calU$} to be the collection of the stars of all these sets with
$\calU$:
\[\calA * \calU = \{A * \calU \colon A \in \calA\}.\]
If $\calV, \calU$ are collections of sets, we say $\calV$ is
\define{generated from}
$\calU$ if $\calV = \calA * \calU$ for some $\calA \subset \powset(X)$. In particular,
every set in $\calV$ is a union of sets from $\calU$.
\end{defn}

So for example, if $\calU$ is a cover of $X$, then for all $A \subset X$,
$A \subset A * \calU$. And so if $\calV \subset \powset(X)$ is a cover of
$X$, then $\calV * \calU$ is also a cover of $X$. We also note here how the
star operation affects the diameter and separation of sets, since we'll
want to calculate how it changes the parameters $s$ and $d$ of
$(s,d)$-covers. If sets in $\calU$ have diameter bounded by $d$, then
$\diam(A * \calU) \leq \diam(A) + 2d$, and if $A, B \subset X$ have
$\rho(A,B) > s$, then $\rho(A * \calU,B* \calU) > s - 2d$. 

Following \cite[Section 4]{CJMSTD}, we also use the notion of a family of
sets dividing a pair:

\begin{defn}
  If $\{x,x'\} \subset X$ and $\calF \subset \powset(X)$ we say that
  $\calF$ \define{divides} $\{x,x'\}$ if there is an $A \in \calF$ such that
  both $A$ and $X \setminus A$ meet $\{x,x'\}$. Similarly, we say $\calU =
  (\calU_0, \ldots, \calU_n)$ divides $\{x,x'\}$ if there is some $i$ so that
  $\calU_i$ divides $\{x,x'\}$. We define $E_{\calU}$ to be the equivalence
  relation on $X$ where $x \mathrel{E_{\calU}} x'$ if and only if
  $\{x,x'\}$ is not divided by $\calU$. Equivalently, $x \mathrel{E_{\calU}} x'$ if and
  only if $x \in A \leftrightarrow x' \in A$ for all $i$ and $A \in \calU_i$.  
\end{defn}

If $\calU = (\calU_0, \ldots, \calU_n)$ is a
cover of $X$, then note that $[x]_{E_{\calU}}$ will be contained in the intersection of all the sets
from the $\calU_i$ that contain $x$. So for example, if $\calU$ is a cover
of $X$ by sets of uniformly bounded diameter, then the $E_{\calU}$ classes
will have uniformly bounded diameter. 

Below we will use two different ways of controlling how dividing is
related in different collections $\calA, \calB$ of sets. If $\calU$ is a
cover of $X$ and $\calA$ is generated by $\calU$, then if $\calA$ divides
$\{x,x'\}$, then $\calU$ must divide $\{x,x'\}$. This is since if $A \in \calA$
and its complement meet $\{x,x'\}$, then since $\calA$ is generated by $\calU$,
there is some $U \in \calU$ with $U \subset A$ that meets $\{x,x'\}$ and the
complement of $U$ must also meet $\{x,x'\}$.
However, if $\calA$ is a
$2r$-separated collection of sets that divides $\{x,x'\}$ and $\rho(x,x')
\leq r$, then $B_r(\calA)$ will not
divide $\{x,x'\}$, since if $A \in \calA$ meets $\{x,x'\}$, then $B_r(A)$ must
contain both points $x,x'$. More generally, any
enlargements of the sets in $B_r(\calA)$ which remain disjoint cannot
divide $\{x,x'\}$.

Now we can give a general lemma showing that if $\rho_0 \geq \rho_1 \geq \ldots$
is a decreasing sequence of locally finite Borel extended metrics with
sufficiently slow dimension growth, then the union of their finite distance
equivalence relations $\bigunion_n E_{\rho_n}$ is hyperfinite. The statement of
this lemma tries to give the most general possible version of this construction.
We will state some simpler corollaries afterwards which are easier to apply in
practice. 

\begin{lem}\label{hyperfinite-construction}
  Suppose $X$ is a standard Borel space, $\rho_0 \geq \rho_1 \geq \ldots$
  is a decreasing sequence of locally finite Borel extended metrics on $X$,
  $(s_n)_{n \in \N}$, $(d_n)_{n \in \N}$ are sequences of positive real
  numbers so that $d_n \to \infty$.

  Define $(d_n')_{n \geq 0}$ recursively by $d_0' = d_0$ and 
  \[d_n' \coloneqq d_n + 4(\dim_B^{\rho_{n+1}}(s_{n+1},d_{n+1}) +2)d_{n-1}' \text{ for $n \geq
  1$,}\]
  and assume that for all $n \geq 1$,
  \[s_n \geq 4 (\dim_B^{\rho_{n+1}}(s_{n+1},d_{n+1}) + 2) d'_{n-1}.\]
  Then $\bigunion_n E_{\rho_n}$ is hyperfinite.
\end{lem}
\begin{proof}

We will take a collection of Borel $(s_n,d_n)$-covers for
$(X,\rho_n)$ and build
from them another collection of $(s_n',d_n')$-covers $\calV^n$ for 
$(X,\rho_n)$, where we define $s_n'$ below, so that if
$\rho_m(x,x') < \infty$ for some $m$, then $\calV^n$ will divide
$\{x,x'\}$ for at most finitely many $n$. Hence, the tail intersections of the equivalence
relations $E_{\calV^n}$ will witness the hyperfiniteness of $\bigunion_n
E_{\rho_n}$.

  Let $c_n = \dim_B^{\rho_n}(s_n,d_n)$, set $s_0'=s_0$, and, for $n \geq 1$,
  set $s_n' = s_n - 4(c_{n+1} + 2) d'_{n-1}$. (The $c_n$ are finite by our
  equations above, dropping an initial term if needed). 
  Note that our equation above on $s_n$ implies that $s_n' \geq 0$.
  For each $n$, fix a Borel
  $(s_n,d_n)$-cover $\calW^n = (\calW^n_0, \ldots, \calW^n_{c_n})$ of
  $(X,\rho_n)$. 
  We will build from each of these covers 
  a 
  %$(c_{n+1} + 2) \times (c_n+1)$ 
  ``matrix'' $\calU^n_{ij}$ of sets $\calU^n_{ij}
  \subset \powset(X)$ where $0 \leq i \leq c_{n+1}
  + 1$ and $0 \leq j \leq c_n$ with the following properties: 
  \begin{enumerate}
  \item The $i$th row $\calU^n_i = (\calU^n_{ij})_{j \leq c_n}$
  will be an $(s_n',d_n')$-cover of $X$ for every $i$ and $n$. 
  \item For every $n \geq 1$, the $j$th column $\calU^n_{*,j} = (\calU^n_{ij})_{i \leq c_{n+1} + 1}$
  will satisfy that if $\rho_n(x,x') \leq d'_{n-1}$, then $\{x,x'\}$ is divided by at
  most one family $\calU^n_{ij}$ in this column. 
  \item For every $n \geq 1$, each $\calU^n_{ij}$ is generated from $\calU^{n-1}_{j}$.
  \end{enumerate}

  We'll define the sets $\calU^n_{ij}$ by recursion
  on $n$ and verify their properties (1)-(3) inductively. For the base case, let $\calU^0_{ij} = \calW^0_j$ for
  every $0 \leq i \leq c_1 + 1$ and $0 \leq j \leq c_0$. So we are copying
  the same cover in every row of the $0$th matrix. Item (1) then is
  satisfied, and properties (2) and (3) are vacuously true.

  For $n \geq 1$, we first define the $i = 0$ row by:
  \[\calU^n_{0j} \coloneqq \calW^n_j * (\calU^{n-1}_{j}),\]
  so $\calU^n_0$ is a Borel $(s_n - 2 d_{n-1}',d_n + 2d'_{n-1})$-cover and hence
  also a Borel $(s_n - 4 d_{n-1}',d_n + 4d'_{n-1})$-cover.
  For subsequent rows $0 < i \leq c_{n+1} + 1$, we take a ball of
  radius $d'_{n-1}$ around the previous row and then star with $\calU^{n-1}_j$. 
  \[\calU^n_{ij} \coloneqq B_{d'_{n-1}}(\calU^n_{(i-1)j}) * (\calU^{n-1}_{jk})_{k
  \leq c_{n-1}}.\]
  So inductively, for each $i$, $\calU^n_i$ is a $(s_n - 4(i + 1)d'_{n-1}
  , d_n + 4(i+1)d'_{n-1})$-cover of $X$, since at each step we expand
  around each set by a ball of radius $d_{n-1}'$, and then apply the star
  operation with a cover with sets of diameter at most $d_{n-1}'$ which
  both enlarges the diameter and shrinks the separation of our sets by
  $4d'_{n-1}$. Then item (3) is clearly true. 

  All that remains is to check property (2). Suppose $\rho_n(x,x') \leq d'_{n-1}$,
  and $\{x,x'\}$ is divided by $A \in \calU^n_{ij}$. Then both $x,x' \in
  B_{d'_{n-1}}(A)$, and thus $\{x,x'\}$ is not divided by $\calU^n_{i'j}$ for all
  $i' > i$, since both $x,x'$ will be contained in some set in $\calU^n_{i'j}$,
  and since $\calU^n_{i'j}$ is a collection of disjoint sets by the
  assumption that $s_n' \geq 0$.

  We now have the following claim:

  \emph{Claim:}
    Suppose $(\calU^n_{ij})_{i \leq c_{n+1}+1,j \leq c_n}$ is a collection of
    sets with the above properties and suppose $\calV^n = \calU^n_{c_{n+1} + 1}$ is the
    $(s_n',d_n')$-cover from the last row of each matrix. Then if $\rho_m(x,x')
    < \infty$ for some $m$, then there are at most finitely many $n$ so that 
    $\calV^n$ divides $\{x,x'\}$.

    To prove the claim, fix $k$ and $(x,x')$ so that $\rho_k(x,x') <
    \infty$. Since $d_n \to \infty$ and hence
    $d_n' \to \infty$, we can choose $m \geq k$ so large that
    $\rho_k(x,x') < d'_{n-1}$ for every $n>m$. Then 
    $\rho_n(x,x') < d'_{n-1}$ whenever $n>m$, since
    $\rho_n \leq \rho_k$ for $n\geq k$.

    Now assume $n > m$.
    By assumption (2), for each $j$, $\{x,x'\}$ can be divided by at most one
    family $\calU^n_{ij}$ in the $j$th column.
    Furthermore, if $\{x,x'\}$ is divided by some $\calU^n_{ij}$ in the $j$th
    column of the $n$th matrix, where $j \leq c_n$, then it
    is also divided by some $\calU^{n-1}_{jk}$ in the $j$th row $\calU^{n-1}_{j}$
    of the $(n-1)$st matrix since each $\calU^n_{ij}$ in the $j$th column is
    generated by $\calU^{n-1}_j$ by assumption (3) and so by
    the observation above. 
    Hence, there is an injection from those pairs of indices $(i,j)$ for which
    $\calU^n_{ij}$ divides $\{x,x'\}$
    into the indices $(i,j)$ from the previous matrix for which $\calU^{n-1}_{ij}$
    divides $\{x,x'\}$. This injection
    never maps into the last row of the $n-1$st matrix $\calU^{n-1}_{ij}$ since
    the columns of $\calU^n_{ij}$ are indexed by $j \leq c_n$, but 
    the rows of $\calU^{n-1}$ are indexed by $i \leq c_n + 1$.

    Now we count the total number of divisions. Let
    $u_n = |\{(i,j) \colon \calU^n_{ij} \text{ divides } \{x,x'\}\}|$
    be the number of indices so that $\calU^n_{ij}$ divides $\{x,x'\}$ and let
    $
    v_n = |\{j \colon \calU^n_{c_{n+1}+1,j} \text{ divides } \{x,x'\}\}|
    $
    be the number of indices in the last row of this matrix dividing
    $\{x,x'\}$. 
    So $u_n \geq v_n$ for every $n$. Note $v_n
    = 0$ if and only if $\calV^n$ does not divide $\{x,x'\}$.
    For $n\geq m$, 
    our injection defined above gives $u_n-v_n\geq u_{n+1}$, and so by induction
    $u_m \geq u_{n+1} + \sum_{k = m}^{n} v_k$. Therefore
    $\sum_{k=m}^{\infty}v_k\leq u_m<\infty$. Hence 
    $v_k=0$ for all but finitely many $k$, and so
    $\calV^n$ divides $\{x,x'\}$ for only finitely many $n$.
    \qed Claim

  So if $\rho_m(x,x') < \infty$, then $x \mathrel{E_{\calV^n}} x'$
  for all but finitely many $n$. 
  Hence the equivalence relations
  $F_m \coloneqq \bigcap_{n \geq m} E_{\calV^n}$ are increasing and witness the
  hyperfiniteness of $\bigunion_n E_{\rho_n}$. 

\end{proof}

\begin{remark}\label{rem:locally-countable}
  We have proved the above lemma for locally finite Borel extended metric spaces.
  However, we remark that the same proof works unchanged for locally countable
  Borel extended metric spaces, if we modify the definition of $\dim_B^X(s,d)$
  to require not only that the collections of sets in our
  $(s,d)$-covers be uniformly bounded, but also that they be smooth.
  Suppose $(X,\rho)$ is a locally countable Borel extended metric space. Then
  say that a \define{smooth Borel $(s,d)$-cover} of $X$ is an $(s,d)$-cover
  $\calU = (\calU_0, \ldots, \calU_n)$ where each $\calU_i$ is Borel, and the
  equivalence relation $E_{\calU_i}$ is smooth. Then define the smooth
  Borel $2$-parameter dimension function of $X$ by:
  \[\dim^X_B(s, d) \coloneqq \inf\{n \in \N \colon
    \text{there is a smooth Borel $(s,d)$-cover of $X$ with $n+1$
    elements}\}.\]
  With this definition of $\dim^X_B(s, d)$ for locally countable spaces, we note that 
  \Cref{hyperfinite-construction} and all the rest of the results in this
  section are true for all locally countable Borel
  extended metric spaces. 
  This follows from the fact that a CBER is hyperfinite iff it is
  hypersmooth \cite[Theorem 5.1]{DJK}, and that the operations of taking balls and the star operations
  preserve smoothness. Precisely, if $\calU = (\calU_0, \ldots, \calU_n)$ is a smooth
  Borel $(s,d)$-cover, then $B_r(\calU_0), \ldots, B_r(\calU_n)$ is a smooth
  Borel $(s-2r,d + 2r)$-cover for any $r \leq s/2$. Furthermore if $\calV \subset \powset(X)$ is a
  smooth Borel $s'$-separated collection of sets and $s' \geq 2d$, then $\calV *
  \calU$ is also a smooth Borel $s' - 2d$-separated collection of sets. The
  smoothness of these collections follows from Lusin-Novikov uniformization.

  This same remark that our results can be generalized to locally countable
  metrics applies to \Cref{easier-hyperfinite} and
  \Cref{power-function-hyperfinite} below.
\end{remark}

A typical way of verifying the hypotheses of
\Cref{hyperfinite-construction} is to arrange that $d'_n \leq 2 d_n$.
In that case, we have the following simplified criterion for checking
hyperfiniteness. We also include another simplifying assumption that our
recurrence about $\dim_B^{\rho_m}(s_{n+1},d_{n+1})$ only needs to hold for
all sufficiently large $n$. 

\begin{cor}\label{easier-hyperfinite}
  Suppose $X$ is a standard Borel space,
  $\rho_0 \geq \rho_1 \geq \ldots$ is a decreasing sequence of locally
  countable Borel extended metrics on $X$,
  and $(s_n)_{n \in \N}$, $(d_n)_{n \in \N}$ are sequences of positive real numbers
  with $s_n \to \infty$
  so that for every $m$,
  \begin{equation}
  \label{eh-assumption}
  24\dim_B^{\rho_{m}}(s_{n+1},d_{n+1})d_{n-1} \leq s_n \leq d_n
  \end{equation}
  for all sufficiently large $n$. 
  Then $\bigunion_n E_{\rho_n}$ is hyperfinite.
\end{cor}
\begin{proof}
  Suppose first there are infinitely many $m$ such that
  $\dim_B^{\rho_{m}}(s_{n+1},d_{n+1}) = 0$ for infinitely many $n$. 
  Then
  $\bigunion_n E_{\rho_n}$ is trivially hyperfinite as follows: 
  Choose strictly increasing sequences $(m_k)_{k \in \N}$ and $(n_k)_{k \in
  \N}$ of integers such that $\dim_B^{\rho_{m_k}}(s_{n_k+1},d_{n_k+1}) = 0$, and
  let $\calU^k$ be the corresponding $(s_{n_k + 1}, d_{n_k+1})$-cover of $X$
  which consists only of a single family of sets. Note that $\calU^k$ is $s_{n_k
  + 1}$-separated. Then the equivalence relations $F_n = \bigcap_{k \geq n}
  E_{\calU^k}$ witness the hypersmoothness, and hence the hyperfiniteness, of $\bigunion_n E_{\rho_n}$. This is
  since if $\rho_m(x,x') < \infty$, then for some sufficiently large $n$,
  $\rho_{m_k}(x,x') \leq s_{{n_k} + 1}$ for all $k \geq n$ since the metrics
  $\rho_n$ are decreasing, $n_k$ and $m_k$ are increasing, and $s_n \to \infty$.
  Then $\calU^k$ does not divide $\{x,x'\}$ for all such $k \geq n$, since the
  sets in $\calU^k$ are $s_{n_k + 1}$-separated. Hence $x \mathrel{E_{\calU^k}}
  x'$ for all $k \geq n$.

  So we may assume that for all but finitely many $m$,
  \begin{equation}\label{eh-dim1}  
  \dim_B^{\rho_{m}}(s_{n+1},d_{n+1}) \geq 1 
  \end{equation}
  holds for all sufficiently large $n$.
  Next we use some reindexing to simplify our assumptions. 
  We will replace our decreasing sequence $\rho_0 \geq \rho_1 \geq \ldots$
  of metrics with the sequence $\rho_{i_0} \geq \rho_{i_1} \geq \ldots$ 
  where $(i_n)_{n \in \N}$ is a slow-growing nondecreasing sequence of integers such
  that $i_n \to \infty$ as $n \to \infty$. Note that since 
  the equivalence relations $E_{\rho_n}$ are increasing with $n$, the union
  $\bigunion_{n \geq 0} E_{\rho_n} = \bigunion_{n \geq 0} E_{\rho_{i_n}}$
  does not change. By choosing $i_0$ sufficiently large, removing some finite
  initial segment from the beginning of the sequences $(d_n)_{n \in \N}$ and
  $(s_n)_{n \in \N}$, and choosing $(i_n)$ to be sufficiently slow growing,
  we may assume without loss of generality that both \Cref{eh-dim1} and \Cref{eh-assumption}
  hold for all $n \geq 1$ and $m = n + 1$. 

  Now we will show that the hypotheses of \Cref{hyperfinite-construction} hold.
  We claim that if we define $d'_0 \coloneqq d_0$ and $d'_n \coloneqq d_n +
  4(\dim_B^{\rho_{n+1}}(s_{n+1},d_{n+1}) + 2)d_{n-1}'$, then $d'_n \leq 2 d_n$ for all $n$ by
  induction. The base case is clear. For the induction step, note that 
  \Cref{eh-dim1} implies that 
  $4 (\dim_B^{\rho_{n+1}}(s_{n+1},d_{n+1}) +
  2) \leq 12 \dim_B^{\rho_{n+1}}(s_{n+1},d_{n+1})$.
  Then 
  \[d'_n \leq d_n + 12\dim_B^{\rho_{n+1}}(s_{n+1},d_{n+1})d_{n-1}'\leq d_n + 24
  \dim_B^{\rho_{n+1}}(s_{n+1},d_{n+1})d_{n-1} \leq 2d_n\]
  where at the penultimate step we have used the induction hypothesis.
  Finally, this means that
  \[4(\dim_B^{\rho_{n+1}}(s_{n+1},d_{n+1}) + 2)d_{n-1}' \leq 12
  \dim_B^{\rho_{n+1}}(s_{n+1},d_{n+1}) d_{n-1}' \leq 24
  \dim_B^{\rho_{n+1}}(s_{n+1},d_{n+1}) d_{n-1} \leq s_n\]
  by our assumption.
\end{proof}

\begin{remark}
For the sake of keeping our proofs simple, we have not optimized the constants
in \Cref{hyperfinite-construction} and \Cref{easier-hyperfinite}.
For example,
the factor of $4$ in both of our equations in the statement of \Cref{hyperfinite-construction} could be improved to $2 +
\ve$ for any $\ve$ by replacing the $d'_{n-1}$-balls in the construction above
with $\frac{\ve}{2} d'_{n-1}$-balls. However, this sort of constant-factor improvement
would not yield an improvement to the exponent in Theorem A.   
Similarly, the constant in \Cref{easier-hyperfinite} could be improved from $24$ to $8 + \ve$ for
any $\ve$ by using that improved version of \Cref{hyperfinite-construction}, and by replacing \Cref{eh-dim1} with
the equation $\dim_B^{\rho_{m}}(s_{n+1},d_{n+1}) \geq 1/\ve$ by using the fact
that if this fails for infinitely many $m$, then $\bigunion_{n \geq 0}
E_{\rho_n}$ is an increasing union of spaces of finite asymptotic dimension and
is hence hyperfinite by \cite[Theorem 7.3]{CJMSTD}. However, as
above this constant improvement would yield
no improvement to our exponent in Theorem A.
We also remark that \Cref{easier-hyperfinite} is still true with the weaker assumption that
$24\dim_B^{\rho_{n+1}}(s_{n+1},d_{n+1})d_{n-1} \leq s_n$, and
$24\dim_B^{\rho_{n+1}}(s_{n+1},d_{n+1})d_{n-1}\leq d_n$ (with no relationship between $s_n$
and $d_n$), assuming that $d_n \to \infty$. However in all cases of interest, we
will have that $s_n \leq d_n$, so we have stated the inequalities this way
to make the statement simpler. 
\end{remark}

Recall that in \Cref{dimension-subexp} we proved an estimate of the form 
$\dim^X_B(s, C s^a) \in O(s^b)$ (where $C > 0$ is a constant, and $a =
1/(1-\alpha)$ and $b = \alpha/(1-\alpha)^2$) for graphs of growth $\vol(r) \in
\exp(O(r^\alpha))$. So we are interested in cases where
one can use \Cref{hyperfinite-construction} to prove hyperfiniteness for spaces
with this type of polynomially bounded and controlled dimension growth. 

\begin{thm}\label{power-function-hyperfinite}
  Suppose $a\geq 1$, $b \geq 0$, and $ab\leq 1/4$.
  Suppose $\rho_0 \geq \rho_1 \geq \ldots$ is a decreasing sequence of locally countable Borel extended
  metrics on a standard Borel space $X$ such that for every $m$, there is a
  constant $C_m$ such that 
  \[
    \dim^{\rho_m}_B(s,C_m s^a)\in O(s^b).
  \]
  Then $\bigunion_n E_{\rho_n}$ is hyperfinite.
\end{thm}
\begin{proof}
  Let
  \[
    \beta \coloneqq 2a, \text{ }
     s_n \coloneqq \exp(\beta^n \log n)
     \text{ and } d_n \coloneqq \log \log(s_n) s_n^a
  \]
  for $n \geq 2$.
  \trivial{These formulas are only needed on a tail of the sequences so it is
  fine that they are undefined for $n < 2$.}
  We will verify the hypotheses of \Cref{easier-hyperfinite}.
  So we need to show that for sufficiently large $n$,
  \[
  24\dim_B^{\rho_{m}}(s_{n+1},d_{n+1})d_{n-1} \leq s_n \leq d_n.
  \]
  Clearly $s_n \leq d_n$ for sufficiently large $n$, so to finish it
  suffices to show that 
\[\frac{\dim_B^{\rho_{m}}(s_{n+1},d_{n+1})d_{n-1}}{s_n} \to 0 \text{ as } n
\to \infty.\]

  Note first that for sufficiently large $n$, $d_n \geq C_m s_n^a$, and so 
  for sufficiently large $n$,
  $\dim_B^{\rho_m}(s_n,d_n) \leq \dim_B^{\rho_m}(s_n,C_m s_n^a) \in
  O(s_n^b)$. Note also that $\log \log(s_{n}) = \log(\beta^n \log n) = n \log(\beta)
  + \log \log n \in O(n)$. Hence,
  \[
  \begin{split}
  \frac{\dim_B^{\rho_{m}}(s_{n+1},d_{n+1})d_{n-1}}{s_n} & =
  \frac{O(s_{n+1}^b)\log \log(s_{n-1}) s_{n-1}^a}{s_n} \\
  & = O(n) \frac{ s_{n+1}^b s_{n-1}^a}{s_n} \\
  & \leq O(n) \exp \left( \beta^n  (\frac{1}{2} \log (n+1) + \frac{1}{2} \log
  (n-1) - \log(n)) \right)\\
  & \leq O(n) \exp \left( \frac{\beta^n}{2} \log \left(
  \frac{(n+1)(n-1)}{n^2} \right) \right) \\
  & \leq O(n) \exp \left( \frac{\beta^n}{2} \log \left(
  1 - \frac{1}{n^2} \right) \right) \leq O(n)
  \exp\left(-\frac{\beta^n}{2n^2}\right) \to 0 \text{ as $n \to \infty$}.
  \end{split}
  \]
\end{proof}

\Cref{easier-hyperfinite} and \Cref{power-function-hyperfinite} combined prove
Theorem C from the introduction.

\begin{remark}\label{rem:ab-leq-1/4}
We remark that the assumption that $0<ab\leq 1/4$ in
\Cref{power-function-hyperfinite} is the optimal result where we can obtain
hyperfiniteness by directly applying 
\Cref{hyperfinite-construction} for spaces where 
$\dim^X_B(s,C s^a)$ has growth of order $s^b$. That is, suppose $(X,\rho)$ is a
Borel extended metric space and there are
constants $C, D > 0$ such that $\dim^{\rho}_B(s,d) \geq D s^b$ for all sufficiently
large $s$ and $d \geq C
s^a$. Suppose also 
we have sequences
$(s_n)_{n \geq 0}$ and $(d_n)_{n \geq 0}$ where $d_n \geq C s_n^a$ and $d_n \to \infty$
satisfying the conditions in \Cref{hyperfinite-construction}, where $\rho_n =
\rho$ is the constant sequence.
Then we claim that $4ab \leq 1$.

This is since we must satisfy 
$s_n \geq 4 (\dim_B^{\rho_{n+1}}(s_{n+1},d_{n+1}) + 2) d'_{n-1}$ by
assumption, so we have that
$s_{n+1}^b d'_{n-1} \in O(s_n)$. Then since 
$d'_{n-1}\geq d_{n-1} \geq C s_{n-1}^a$, this implies
$s_{n+1}^b s_{n-1}^a \in O(s_n)$. Let $x_n \coloneqq \log s_n$. Then taking
logarithms gives
\[
  a x_{n-1}+b x_{n+1}\leq x_n+O(1).
\]
Assuming $4ab>1$, choose $0<a'<a$ and $0<b'<b$ such that $4a'b'>1$.
Now $x_n\to\infty$ because $s_n \geq 8 d'_{n-1} \geq 8 d_{n-1}$ and $d_n \to
\infty$ as $n \to \infty$. 
This implies 
$a' x_{n-1}+ b' x_{n+1} < x_n$ for sufficiently large $n$.
By 
the AM-GM inequality, this implies $x_n > 2\sqrt{a'b' x_{n-1} x_{n+1}}$ so $x_n^2 > 4a'b' x_{n-1}
x_{n+1}$, and $\frac{x_n}{x_{n-1}} > 4a'b' \frac{x_{n+1}}{x_n}$. Thus 
there is an $N$ such that $\frac{x_{m+1}}{x_m} \leq \frac{1}{(4a'b')^{m-N}} \frac{x_{N+1}}{x_N}$ for all $m \geq N$. Hence
if $4a'b' >
1$, then $\frac{x_{m+1}}{x_{m}} \to 0$ as $m \to \infty$, 
contradicting $x_n \to \infty$. 
\end{remark}

We can now prove that a decreasing sequence of locally countable Borel
extended metrics of growth at most $\exp(O(r^\gamma))$ is hyperfinite, where
$\gamma$ is as in the abstract:

\begin{thm} \label{thm:uniform-gamma-hyp}
Let $\gamma \approx 0.1523$ be the unique real root of $(1 - \gamma)^3 - 4
\gamma = 0$, and suppose $X$ is a standard Borel space and
$\rho_0 \geq \rho_1 \geq \ldots$ is a decreasing
sequence of locally countable Borel extended metrics on $X$ so that for every $n$, $\rho_n$ has
volume growth bounded by 
$\vol^{\rho_n}(r) \in \exp(O(r^{\gamma}))$. 
Then
$\bigunion_n E_{\rho_n}$ is hyperfinite.
\end{thm}
\begin{proof}
  By \Cref{dimension-subexp}, for each $n$ there is a constant $C_n > 0$
  such that $\dim^{\rho_n}_B(s,C_n s^{1/(1-\gamma)}) \in O(s^{\gamma/(1-\gamma)^2})$, and hence by
  \Cref{power-function-hyperfinite}, it is hyperfinite if $4 \gamma \leq (1 - \gamma)^3$, which holds by the definition of $\gamma$.
\end{proof}

We can remove the uniform volume growth assumption and weaken it to a
pointwise volume growth bound by passing to a new decreasing sequence of metrics
to which we can apply \Cref{thm:uniform-gamma-hyp}. This gives Theorem A from
the introduction:

\begin{thm} 
Let $\gamma$ be the unique real root of $(1 - \gamma)^3 - 4
\gamma = 0$, so $\gamma \approx 0.1523$.
Suppose $X$ is a standard Borel space and
$\rho_0 \geq \rho_1 \geq \ldots$ is a decreasing
sequence of Borel extended metrics on $X$ so that for every
$n$, $\rho_n$ has pointwise 
volume growth $\exp(O(r^{\gamma}))$. That is, for every $n \in \N$ and $x \in
X$, there exist $C,N > 0$ such that $|B^{\rho_n}_r(x)| \leq
\exp(C r^\gamma)$ for all $r \geq N$. Then $\bigunion_n E_{\rho_n}$ is
hyperfinite, where $E_{\rho_n}$ is the finite distance equivalence relation for
$\rho_n$.
\end{thm}
\begin{proof}
  We will define a new decreasing sequence of Borel extended metrics $\rho_n'$ with uniformly bounded volume growth
  $\exp(O(r^\gamma))$ so that the union of the connectedness relations is
  unchanged: $\bigunion_n E_{\rho_n'} = \bigunion_n E_{\rho_n}$, and then apply
  \Cref{thm:uniform-gamma-hyp} to these new metrics $\rho_n'$.

  For $m,n\in\N$, let $A_{m,n}$ be the set of $x \in X$ such that
  $|B_r^{\rho_m}(x)|\leq \exp(nr^\gamma)$ for every $r \geq n$. So the sets
  $A_{m,n}$ are Borel and $\bigunion_{n \in \N} A_{m,n}=X$. Note that the
  restriction of $\rho_m$ to $A_{m,n}$ has uniformly bounded volume growth 
  $\exp(O(r^\gamma))$ for every $n$. Hence, for a single fixed
  extended metric $\rho_m$, we could show that $E_{\rho_m}$ is hyperfinite by applying
  \Cref{thm:uniform-gamma-hyp} to the decreasing sequence of metrics $\sigma_0 \geq
  \sigma_1 \geq \ldots$ defined by setting $\sigma_n(x,x')\coloneqq \rho_m(x,x')$ if $x,x' \in
  A_{m,n}$, $\sigma_n(x,x') = 0$ if $x = x'$, and $\sigma_n(x,x') = \infty$ otherwise. To handle the countably
  many metrics $\rho_n$, we will use a more elaborate version of this sort of idea. 

  For $n\in\N$, define the metric $\rho'_n(x,y)$ to be the infimum of 
  $\sum_{i<k}\rho_{m_i}(x_i,x_{i+1})$ for all sequences $x=x_0,\ldots,x_k=y$ and
  $m_0,\ldots,m_{k-1}\leq n$ such that $x_i,x_{i+1}\in A_{m_i,n}$ for all $i <
  k$.  
  So $\rho_n'$ is the infimum of the lengths of ``labeled'' paths from $x$ to $y$, where each
  step from $x_i$ to $x_{i+1}$ is labeled by $m_i$ where both $x_i, x_{i+1} \in
  A_{m_i,n}$, and we
  measure distance using the metric $\rho_{m_i}$. Note that
  $\rho'_{n+1}\leq\rho'_n$ for all $n$ by definition, and the union of the finite distance relations is
  unchanged: $\bigunion_n E_{\rho_n'} = \bigunion_n E_{\rho_n}$ as required,
  since $\bigunion_{n \in \N} A_{m,n} = X$.

  We now compute a uniform upper bound on the volume growth of the metrics
  $\rho_n'$. We begin by noting there is a
  finite bound on the
  number of steps we need to consider in these walks defining $\rho_n'$. Suppose an
  allowed walk $x_0, \ldots, x_k$ has two steps $i < j < k$ where $m_i = m_j$
  and every step strictly between them has label less than $m_i$.
  Then the part of the walk from $x_i, \ldots, x_{j+1}$ can be replaced by a
  single step directly from $x_i$ to $x_{j+1}$ with label $m_i$. Thus, we may
  assume all walks contain their largest label $m_i$ at most once. Thus, by
  induction (splitting the walk before and after this label),
  we may consider only walks with at most $S(n)$ steps, where $S(n)$ satisfies the
  recurrence $S(0) = 1$ and $S(n) = 2S(n-1)+1$, since we have $1$ step of
  largest label in the walk, and then the walk before and after that has length
  at most $S(n-1)$. 

Now fix $n$ and suppose $r \geq n$. If $y \in B_r^{\rho_n'}(x)$, then
there is such a walk $x=x_0, \ldots, x_k=y$ with $k \leq S(n)$ steps,
labels $m_0, \ldots, m_{k-1}$, and total length at most $r$. Each
step has distance $\rho_{m_i}(x_i,x_{i+1})
\leq r$ in the metric $\rho_{m_i}$. Since
$x_i \in A_{m_i,n}$, we have $|B^{\rho_{m_i}}_r(x_i)| \leq
\exp(nr^\gamma)$. Since
there are $n+1$ possible choices for the value of $m_i$, at each step of the
walk, and then at most $|B^{\rho_{m_i}}_r(x_i)|$ choices of the next point to
walk to given this label $m_i$, we therefore have
 \[|B_r^{\rho'_n}(x)| \leq \left((n+1)\exp(n r^\gamma) \right)^{S(n)}
   \in \exp(O(r^\gamma)).\]
Here $n$ is fixed.
\end{proof}

We remark that the method of making metrics with uniformly bounded volume growth in the above proof is not special to this sort of
intermediate growth; it just uses that the growth bounds we are using have
sufficient closure. In particular, that class of functions
$\exp(O(r^\gamma))$ is closed under taking constant powers and constant multiples.

\section{F{\o}lner tilings and layerings}
\label{sec:folner-tiling}

In this section, we will show that we can use F{\o}lner layerings to obtain
F{\o}lner tilings. We begin by recalling the definition of a F{\o}lner tiling:

\begin{defn}
  Let $r,\ve > 0$.
  An \textbf{$(r,\ve)$-F{\o}lner tiling} of an extended metric space $X$
  is a cover $\calU$ of $X$ by disjoint $(r,\ve)$-F{\o}lner sets of uniformly
  bounded diameter. An \textbf{$\ve$-F{\o}lner tiling} is a $(1,\ve)$-F{\o}lner
  tiling.
\end{defn}
Equivalently, an $(r,\ve)$-F{\o}lner tiling is a uniformly
bounded-diameter equivalence relation on $X$ whose classes are
$(r,\ve)$-F{\o}lner.

Recall also the following definition of a setwise F{\o}lner metric space:

\begin{defn}[\cite{ET25}]
A locally finite extended metric space $(X,\rho)$ is \define{setwise
F{\o}lner} if for every $r,\ve > 0$, there exists some $R > 0$ such that
for every nonempty finite set $A \subset X$ there is an $(r,\ve)$-F{\o}lner set $F$
such that $A \subset F \subset B_R(A)$. 
\end{defn}

Recall that $X$ is \define{F{\o}lner} if for every $r, \ve > 0$, there exists $R
> 0$ such that for every $x \in X$ there is an $(r,\ve)$-F{\o}lner set $F$ so
that $x \in F \subset B_R(x)$. We note for context that Elek and Tim\'ar
\cite{ET25} have shown that every bounded degree graph is F{\o}lner if and
only if it is setwise F{\o}lner, and their proof extends straightforwardly
to show that an extended metric space of uniformly bounded volume growth is
F{\o}lner if and only if it is setwise F{\o}lner. However, we will not need
this fact for any of our arguments.

Metric spaces of subexponential growth are well-known to be setwise
F{\o}lner by an easy argument similar to
\Cref{prop:subexp-Folner-balls}:
\begin{prop}\label{prop:subexp-setwise-Folner}
  Suppose $(X,\rho)$ is an extended metric space of uniformly
  subexponential volume growth. Then $(X,\rho)$ is setwise F{\o}lner. 
\end{prop}
\begin{proof}
  Fix $r,\ve > 0$.
  By subexponential volume growth, choose a positive integer $t$ such that $\vol(r(t + 1)) \leq (1 + \ve)^{t+1}$.
  Since $B_r(A) = \bigunion_{x \in A} B_r(x)$, we have that
  \[
    \frac{|B_r(A)|}{|B_0(A)|}
    \frac{|B_{2r}(A)|}{|B_r(A)|}
    \cdots
    \frac{|B_{(t+1)r}(A)|}{|B_{tr}(A)|}
    = \frac{|B_{(t+1)r}(A)|}{|B_0(A)|}
    \leq \frac{\vol(r(t+1)) |A|}{|A|}
    \le (1 + \ve)^{t+1},
  \]
  so there is some $s \le t$ such that
  $\frac{|B_{r(s+1)}(A)|}{|B_{rs}(A)|} \le 1 + \ve$. So $F = B_{rs}(A)$ is the desired
  $(r,\ve)$-F{\o}lner set since $B_r(F) = B_r(B_{rs}(A))\subset B_{r(s +1)}(A)$.
\end{proof}

To finish, we show that if $X$ is setwise F{\o}lner and has Borel
F{\o}lner layerings, then it has Borel F{\o}lner tilings: 

\begin{lem}\label{packing-and-hull-implies-tiling}
  Suppose $(X,\rho)$ is a locally finite Borel extended metric space of
  uniformly bounded volume growth that is setwise F{\o}lner, and suppose $X$ has
  a Borel $(r,\ve)$-F{\o}lner layering for every
  $r,\ve > 0$. Then $X$ has a Borel $(r,\ve)$-F{\o}lner tiling for every $r,\ve >
  0$.
\end{lem}
\begin{proof}
Suppose $r,\ve > 0$, and define $\delta \coloneqq \ve/(1 + \vol(r))$. We will
construct a $(r,\ve)$-F{\o}lner tiling.

Since $X$ is
setwise F{\o}lner, choose a sufficiently large $R$ so that for every nonempty
finite $A$, there is a $(r,\delta)$-F{\o}lner set $\phi(A)$ such that
$A \subseteq \phi(A) \subseteq B_R(A)$. 
By Lusin-Novikov uniformization, we can choose such a
function $\phi$ such that $\phi$ is Borel, since there are 
finitely many $(r,\delta)$-F{\o}lner sets satisfying
$A \subseteq F \subseteq B_R(A)$. Without loss of generality we may assume
$R \geq r/2$.

Let $\calF = (\calF_0, \ldots, \calF_{n-1})$ be a
$(2R,\delta)$-F{\o}lner layering of $X$.
The sets $\{\phi(F) \colon F \in \calF_i \text{ for some $i<n$}\}$ are disjoint, and are all
$(r,\delta)$-F{\o}lner by definition of $\phi$. To finish the proof,
we will enlarge each of these sets a small amount to make a tiling that
covers all of $X$, while keeping the property that these new enlarged sets are
$(r,\ve)$-F{\o}lner.

Now $X \setminus \supp(\calF)$ is covered by $\{\bndry_{2R}^{X_i}(F) \colon F \in
\calF_i \land i<n\}$.
If $A$ is $(r,\delta)$-F{\o}lner and $|B| \leq \delta |A|$, then $A \cup B$ is
$(r,(\vol(r)+1)\delta)$-F{\o}lner, since $|A| \leq |A \cup B|$, and
$|\bndry_r(A \cup B)| \leq |\bndry_r(A)| + |\bndry_r(B)| \leq \delta
|A| + \vol(r) \delta |A|$. So we can finish by enlarging the sets $\phi(F)$ to
include part of $\bndry_{2R}^{X_i}(F)$ so that they form a tiling, and these
sets will still be $(r,\ve)$-F{\o}lner.

Precisely, let $\psi$ be a Borel function on
$X \setminus \supp(\{\phi(F) \colon F \in \calF_i \text{ for some $i<n$}\})$
such that if $\psi(x) = (i,F)$, then $F \in \calF_i$ and
$x \in \bndry^{X_i}_{2R}(F)$. So $\psi$ selects for each remaining $x$ some
$F \in \calF_i$ whose $2R$-boundary contains it. Then our desired tiling is
$\{\phi(F) \cup \{x \colon \psi(x) = (i,F)\} \colon F \in \calF_i \text{ for some $i<n$}\}$.
\trivial{For $F \in \calF_i$, put
$C_F \coloneqq \{x \colon \psi(x)=(i,F)\}$. Then
$C_F \subseteq \bndry_{2R}^{X_i}(F)$, so
$|C_F| \leq \delta|F| \leq \delta|\phi(F)|$. Applying the preceding estimate
with $A=\phi(F)$ and $B=C_F$ shows that $\phi(F)\union C_F$ is
$(r,(\vol(r)+1)\delta)=(r,\ve)$-F{\o}lner. The sets $\phi(F)$ are disjoint,
and $\psi$ assigns every remaining point to exactly one of them, so these
enlarged sets form a disjoint Borel cover of $X$. They have uniformly bounded
diameter since $\phi(F) \subseteq B_R(F)$ and
$C_F \subseteq B_{2R}(F)$.}
\end{proof}

Putting together \Cref{cor:subexp-layering}, \Cref{prop:subexp-setwise-Folner} and
\Cref{packing-and-hull-implies-tiling} proves Theorem B from the introduction:

\begin{thm}
  Suppose $(X,\rho)$ is a Borel extended metric space of uniformly
  subexponential volume growth. Then for every $r,\ve >
  0$, there is a Borel $(r,\ve)$-F{\o}lner tiling of $(X,\rho)$.\qed
\end{thm}

\begin{remark}
  We remark that the proof of \Cref{packing-and-hull-implies-tiling} can be
  easily modified to show analogous clopen tilings and topological almost
  finiteness results in the setting of
  topological dynamics, since our constructions are entirely ``local''. The modifications of our constructions in \Cref{layering-existence} and \Cref{packing-and-hull-implies-tiling} required to prove the topological dynamics analogues exactly follow
  similar modifications
  discussed in \cite[Section 10]{CJMSTD}, showing that many of the constructions
  of that paper easily adapt to the realm of topological dynamics.
\end{remark}

\subsection{An equivalence between layerings and tilings}
\label{subsec:tiling-layer-equiv}

To finish this section, we expand on the above, and prove an equivalence
between tilings and layerings in arbitrary Borel extended metric spaces of
uniformly bounded volume growth.

We can now show that the existence of Borel F{\o}lner tilings and
F{\o}lner layerings are equivalent:

\begin{thm}
  Suppose $X$ is a Borel extended metric space of uniformly
  bounded volume growth. Then $X$ has
  a Borel $(r,\ve)$-F{\o}lner tiling for every $r,\ve > 0$ iff $X$ is setwise
  F{\o}lner and has a
  Borel $(r,\ve)$-F{\o}lner layering for every $r,\ve > 0$.
\end{thm}
\begin{proof}
We have already proved the reverse implication in
\Cref{packing-and-hull-implies-tiling}. So we prove the two forward
implications. Suppose $X$ has a Borel $(r,\ve)$-F{\o}lner tiling for every
$r,\ve > 0$. Given such a Borel $(r,\ve)$-F{\o}lner tiling $\calF$ of $X$, of diameter bounded
by $d$, if $A \subset X$ is nonempty and finite, let $F = A * \calF$ be the union of the
sets in $\calF$ meeting $A$. Then $F$ is $(r,\ve)$-F{\o}lner since it is a union
of disjoint $(r,\ve)$-F{\o}lner sets, and $F \subset B_d(A)$. So $X$ is setwise
F{\o}lner. 
So it remains to prove the existence of F{\o}lner layerings.

Fix $r,\ve>0$, and choose $\delta$ sufficiently small so that $\frac{1 +
\delta}{1 - \vol(r) \delta} \leq 1 + \ve$.
Let $\calA$
be a Borel $(r,\delta)$-F{\o}lner tiling of $X$. Fix $d$ bounding the diameters
of the members of $\calA$. The Borel graph on $\calA$ in which distinct
$A,A'\in\calA$ are adjacent when $\rho(A,A')\leq r$ has bounded degree (in
particular bounded by $\vol(d + r)-1$), and hence has a Borel $n$-coloring for
some $n$. Thus, we can decompose $\calA$ into sets
\[
  \calA=\calA_0\sqcup\cdots\sqcup\calA_{n-1}
\]
such that each $\calA_i$ is $r$-separated.

We now recursively define a $(r,\ve)$-F{\o}lner layering $\calF_0,\ldots,\calF_{n-1}$ of
$X$. Define $X_0=X$ and, for $i<n$,
  $\calF_i\coloneqq
  \{A\cap X_i:A\in\calA_i\text{ and }A\cap X_i\neq\emptyset\}$ and
  $X_{i+1}\coloneqq X_i\setminus B_r(\supp(\calF_i))$.
So each $\calF_i$ has diameter bounded by $d$ and is $r$-separated.

We verify that for each $A\in\calA_i$, the set $F \coloneqq A \cap X_i$
is $(r,\ve)$-F{\o}lner. Now $A \setminus F = A \setminus X_i = A \inters
\bigunion_{j<i} B_r(\supp(\calF_j))
  \subset B_r^X \left(\bndry_r^X(A)\right)$
  and so 
\[
  |F|\geq |A|-\vol(r)|\bndry_r^X(A)|
  \geq (1-\vol(r)\delta)|A|.
\]
Since $|B_r^X(F)| \leq |B_r^X(A)| \leq (1 + \delta)|A| \leq \frac{1 +
\delta}{1 - \vol(r) \delta} |F|$,
we have that $F$ is $(r,\ve)$-F{\o}lner in $X$, and hence $(r,\ve)$-F{\o}lner in
$X_i$.

To finish, note that $B_r\left(\supp(\calF_0,\ldots,\calF_{n-1})\right)=X$, since if $x
\in X$, then $x \in A$ for some $A \in \calA$, and let $i$ be such that $A \in
\calA_i$. Then either $x \in A \inters X_i$ and hence $x \in \supp(\calF_i)$ by
definition of $\calF_i$, or
$x \in A \setminus X_i = A \inters \bigunion_{j<i} B_r(\supp(\calF_j))$, so
$x \in B_r(\supp(\calF_0, \ldots, \calF_{n-1}))$.
\end{proof}

\section{Separation index growth, subexponential dimension growth, and
\texorpdfstring{$\mu$}{mu}-hyperfiniteness}
\label{sec:mu-hyperfiniteness}

Now we turn to the setting of CBERs on standard probability
spaces, and we begin translating a standard characterization of
$\mu$-hyperfiniteness of locally finite Borel graphs into a characterization of
$\mu$-hyperfiniteness of the equivalence relations $E_\rho$ of locally finite
Borel extended metric spaces $(X,\rho)$, where $\mu$ is a Borel probability
measure on $X$. Recall that we
say that a CBER $E$ is $\mu$-hyperfinite if and only if there is a conull Borel
set $X' \subset X$ such that $E \restriction X'$ is hyperfinite. Equivalently,
if there is a conull $E$-invariant Borel set $X'' \subset X$ such that $E
\restriction X''$ is hyperfinite.
We will then use
our characterization below to prove that if $E_\rho$ is not $\mu$-hyperfinite,
then $(X,\rho)$ has exponential $\mu$-measurable separation index growth.

If $X$ is a locally finite Borel extended metric space and $\mu$ is a finite Borel measure on
$X$, we say that $X$ is an \define{$(s,\ve)$-expander} with respect to $\mu$ if for every
Borel $2s$-separated family
$\calA \subset \powset(X)$ of finite sets, we have $\mu(B_s(\supp
\calA)) \geq (1 + \ve) \mu(\supp \calA)$.

We have the following characterization of $\mu$-hyperfiniteness for locally
finite extended metric spaces, which is an easy adaptation of similar well-known lemmas for group actions and Borel graphs. Item (3) in this characterization is due to 
Elek \cite{E12}, refining an earlier result of Kaimanovich \cite{Kai}. 

\begin{lem}
  \label{lem:mu-hyperfinite-equiv}
Suppose $\rho$ is a locally finite Borel extended metric on a standard probability space
$(X,\mu)$. Then the following are equivalent:
\begin{enumerate}
\item $E_\rho$ is $\mu$-hyperfinite 
\item For every $\ve > 0$ and every $s >
0$, there is a $s$-separated Borel family $\calA \subset \powset(X)$ of finite sets
such that $\mu(\supp(\calA)) \geq 1 - \ve$.
\item For every Borel subset $X' \subset X$ of positive measure and every $\ve >
0$ and $s > 0$, $X'$ is not an $(s,\ve)$-expander.
\end{enumerate}
\end{lem}
\begin{proof}
(1) $\Rightarrow$ (3):
Suppose $(F_n)_{n \in \N}$ are finite Borel equivalence relations witnessing
that $E_\rho$ is hyperfinite on a conull set. Without loss of generality, we may
assume $X'$ is contained in this conull set. By restricting to $X'$, these
equivalence relations also witness that $E_{\rho \restriction X'}$ is
hyperfinite. Pick $r > 2s$, and for each $n$, let 
\[\calA_n \coloneqq \{\{y \in X' \colon B^{X'}_{r}(y) \subset [x]_{F_n}\} \colon x \in X'\}.\]
So $\calA_n$ consists of the $r$-interiors of the $F_n \restriction X'$-classes. Now the
collection $\calA_n$ is $r$-separated and hence has separation greater than
$2s$, since each $F_n$-class meets at most one element $A \in \calA_n$, and
$B_r(A)$ is contained in this $F_n$-class. Every $x \in X'$ is in
$\supp(\calA_n)$ for some $n$, since $B^{X'}_{r}(x)$ is a finite set of points
which must all eventually be $F_n$-related to $x$. Hence $\mu(\supp(\calA_n))
\to \mu(X')$ as $n \to \infty$, and so taking $\calA_n$ with
$\mu(\supp(\calA_n)) > \frac{\mu(X')}{1 + \ve}$ witnesses that $X'$ is not an
$(s,\ve)$-expander. 

(3) $\Rightarrow$ (2):
Fix $s > 0$ and $\ve > 0$, and choose $\delta > 0$ such that $\delta/(1+\delta) < \ve$.
Suppose that for all positive measure sets $X' \subset X$, $X'$ is not an $(s,\delta)$-expander. 
Let $\calA$ be a maximal (modulo $\mu$-null sets) Borel family of $s$-separated finite sets such that
$\mu(B_s(\supp(\calA))) \leq (1 + \delta) \mu(\supp(\calA))$. Such a
collection must exist by a measure exhaustion argument\footnote{
Recall the principle of measure exhaustion: there cannot be uncountably many
disjoint sets of positive measure in a Borel probability space (see e.g.
\cite[Section 215]{F03}). This is true because any uncountable sequence of positive real
numbers has infinite sum; otherwise it would contain only finitely many
numbers greater than $1/n$ for each positive integer $n$.
So any wellordered chain of Borel sets that is strictly increasing in
measure in a
$\sigma$-finite measure space must be countable, and so its union is Borel. 
}. We claim that
$B_s(\supp(\calA))$ is conull, and hence $(2)$ is satisfied. Otherwise, let $X' = X \setminus B_s(\supp(\calA))$. Then we can find some Borel
$s$-separated family of finite sets $\calA' \subset \powset(X')$ whose support has
positive measure such that
$\mu(B_s^{X'}(\supp(\calA'))) \leq (1 + \delta) \mu(\supp(\calA'))$. Note by the definition of $X'$, every point of $\supp(\calA')$ has distance greater than $s$ from $\supp(\calA)$. So
$B_s(\supp(\calA' \union \calA)) \subset B_s(\supp(\calA)) \union B_s^{X'}(\supp(\calA'))$, and
$\calA' \union \calA$ is $s$-separated, contradicting the maximality of $\calA$.

(2) $\Rightarrow$ (1): For each $n$, let $\calA_n$ be a Borel $2n+1$-separated family
such that $\mu(\supp(\calA_n)) \geq 1 - \frac{1}{2^n}$, and let
$\calA'_n = B_n(\calA_n)$, so $\calA'_n$ is $1$-separated. Then let $E_n$
be the equivalence relation whose classes are the elements of $\calA'_n$,
combined with singletons for every $x \notin \supp(\calA'_n)$. 
So $x \mathrel{E_n} y$ if either $x = y$, or
there is an $A \in \calA'_n$ such that $x,y \in A$. Note $E_n$ is a finite Borel
equivalence relation on $X$.

Define $F_n = \bigcap_{k \geq n} E_k$ to
be their tail intersections, so $F_n$ is an increasing sequence of finite
Borel equivalence relations. The equivalence relations $F_n$
will witness that $E_\rho$ is $\mu$-hyperfinite by a Borel-Cantelli argument. Fix $N$. For all $n \geq N$, the $\mu$-measure of the set
of $x$ such that there is some $y \in B_N(x)$ such that $x
\not\mathrel{E_n} y$ is at most $\frac{1}{2^n}$. Hence, by the union bound, for all $m \geq N$, the
measure of the set of $x$ such that there is some $y \in B_N(x)$ such that
$x \not\mathrel{F_m} y$ is at most $\sum_{n=m}^\infty \frac{1}{2^n} = \frac{1}{2^{m-1}}$. Hence, for a conull set of $x$, for all $y \in B_N(x)$ we have 
$x \mathrel{\bigunion_n F_n} y$. Since this is true for every $N$, $(F_n)_{n \in \N}$ witnesses the $\mu$-hyperfiniteness of $E_\rho$. 
\end{proof}

We have the following corollary: if $E_\rho$ is not $\mu$-hyperfinite, then
$\si^X_B$ has exponential growth, and hence so does $\sdim_B^X$. Indeed, we
can characterize $\mu$-hyperfiniteness for locally finite Borel extended metric
spaces via their separation index growth, which extends a characterization due
to Felix Weilacher \cite{Weil} who proved the equivalence of (1) and (2) below.
This is Theorem E from the introduction:

\begin{cor}
  \label{cor:mu-hyperfinite-equiv}
Suppose $\rho$ is a locally finite Borel extended metric on a standard probability space
$(X,\mu)$. Then the following are equivalent:
\begin{enumerate}
\item $E_\rho$ is $\mu$-hyperfinite.
\item $\asi_\mu(X) \leq 1$. 
\item $\si_\mu^X$ has non-exponential growth.
\end{enumerate}
\end{cor}
\begin{proof}
Weilacher \cite{Weil} has shown that (1) implies (2), and
(2) implies (3) since a bounded function has subexponential growth. Hence it suffices to
prove that (3) implies (1), which we do by contraposition. 

  Suppose $E_\rho$ is not $\mu$-hyperfinite. By
  \Cref{lem:mu-hyperfinite-equiv}, there is a Borel subset $X'\subset X$ of
  positive measure and $s,\ve>0$ such that $X'$ is an $(s,\ve)$-expander.
  Clearly $\si_\mu^X(s) \geq \si_\mu^{X'}(s)$ by restricting covers of $X$ to $X'$,
  so it suffices to prove that $X'$ has an exponential lower bound on its separation index growth. 

  Suppose $\calA \subset \powset(X')$ is a family of finite sets that is
  $2ns$-separated. 
  Then since $B^{X'}_{is}(\calA)$ has
  separation greater than $2(n-i)s \geq 2s$ for every $i < n$ and $X'$ is an $(s,\ve)$-expander, we have
  $\mu(\supp(B^{X'}_{(i+1)s}(\calA))) \geq (1 + \ve) \mu(\supp(B^{X'}_{is}(\calA)))$,
  and so $\mu(\supp(B^{X'}_{ns}(\calA))) \geq (1 + \ve)^n
  \mu(\supp \calA)$. Hence $\mu(\supp \calA) \leq \frac{\mu(X')}{(1 +
  \ve)^n}$. Thus, if we cover a conull subset of $X'$ by families of finite sets of separation
  greater than $2ns$, since the total measure of the supports of all the
  families must be at least $\mu(X')$, we must have at least $(1 + \ve)^n$
  families of sets. So $\si_\mu^{X'}(2ns) \geq (1 + \ve)^n-1$ for all integers
  $n$. So $\si_\mu^{X'}(r) > b^r$ for any $1 < b < (1 + \ve)^{1/(2s)}$ for sufficiently large real numbers $r > 0$.
\end{proof}

Note in particular since $\sdim^X_B(s) \geq \si^X_B(s) \geq \si^X_\mu(s)$ for
any Borel probability measure $\mu$, this means that for any locally
finite Borel extended metric $(X,\rho)$ with subexponential Borel dimension
growth, $E_\rho$ is $\mu$-hyperfinite for every Borel probability measure
$\mu$ on $X$. So if $(X,\rho)$ has subexponential Borel dimension growth, but $E_\rho$ is
\emph{not} hyperfinite, then it would solve the hyperfiniteness vs measure
hyperfiniteness problem in the negative \cite[Problem 7.29]{K25}. Stated another way, all
locally finite Borel extended metric spaces $(X,\rho)$ that are known to be
non-hyperfinite have exponential dimension growth. Hence, there is currently a
large gap between these exponential dimension growths in all known
non-hyperfinite spaces, and the dimension growths
which are known to imply hyperfiniteness by Theorem C. 

\section{Borel subexponential dimension growth conjectures}
\label{sec:Borel-dimension}

We begin by restating our conjectures on Borel dimension which would imply a
positive solution to Weiss's question. 

\begin{conjF}
  \mbox{ }
  \begin{enumerate}
    \item 
    Suppose $\Gamma \actson X$ is a Borel action of a finitely
    generated amenable group $\Gamma$ %with finite symmetric generating set $S$
    on a standard Borel space $X$, and $\rho$ %$\rho_S$
    is the word metric on this action.
    Then $(X,\rho)$ has subexponential Borel dimension growth.
  \item If $(X,\rho)$ is a locally finite Borel extended metric space which 
  has subexponential Borel dimension growth, then $E_\rho$ is hyperfinite. 
  \end{enumerate}
\end{conjF}

These conjectures imply a positive answer to Weiss's question, even though part
(1) of the conjecture is only about finitely generated groups. This follows from
a standard way of combining a decreasing sequence of locally finite metrics into
a single locally finite metric.

\begin{prop}\label{conjF-implies-Weiss}
  If Conjecture F is true, then Weiss's question has a positive answer.
\end{prop}
\begin{proof}
Suppose $\Gamma \actson X$ is a Borel action of a
countable amenable group on a standard Borel space. We need to show the action
is hyperfinite. Let $\Gamma = \bigunion_n
S_n$ where $S_0 \subset S_1 \subset \ldots$ are an increasing sequence of finite
symmetric subsets of $\Gamma$, and let $\rho_n$ be the word metric arising from
the restriction of this action to the finitely generated subgroup $\langle S_n\rangle$
generated by $S_n$, so $\rho_0 \geq \rho_1 \geq \ldots$. Then by Conjecture
F.(1) each $\rho_n$ has subexponential dimension growth. 

Now let $f \colon \N \to \N$ be a
sufficiently fast growing increasing function with $f(0)=0$ and $f(n)\to
\infty$, and define 
\[\rho_f(x,x') \coloneqq \inf_n \max(f(n),\rho_n(x,x')).\]
So $\rho_f(x,x') = \max(f(n),\rho_n(x,x'))$ where $n$ is least such that
$\rho_n(x,x') < f(n+1)$.
To see that $\rho_f$ satisfies the triangle inequality, suppose
$\rho_f(x_0,x_1),\rho_f(x_1,x_2)<\infty$, choose $n_0,n_1$ realizing the
corresponding infima, and let $n=\max(n_0,n_1)$. Then
$\rho_n(x_0,x_1)\leq\rho_f(x_0,x_1)$,
$\rho_n(x_1,x_2)\leq\rho_f(x_1,x_2)$, and
$f(n)\leq\max(\rho_f(x_0,x_1),\rho_f(x_1,x_2))$. Hence
\[
\rho_f(x_0,x_2)\leq\max(f(n),\rho_n(x_0,x_2))
\leq\rho_f(x_0,x_1)+\rho_f(x_1,x_2).
\]
So $\rho_f$ is an extended metric, and it has uniformly bounded volume growth. 

By definition of $\rho_f$, and since the metrics $\rho_n$ are decreasing, 
\begin{enumerate}
\item if $\rho_n(x,x') \geq f(n)$, then $\rho_f(x,x') \leq
\rho_n(x,x')$, and 
\item if $\rho_n(x,x') < f(n+1)$, then $\rho_f(x,x') \geq \rho_n(x,x')$.
\end{enumerate}
So if $\calU_0, \ldots, \calU_l$ is an $(s,d)$-cover
of $(X,\rho_n)$, where $s < f(n+1)$ and $d \geq f(n)$, then these sets $\calU_i$
have diameter bounded by $d$ by (1) above, and are $s$-separated by (2). 
Hence $\dim_B^{\rho_f}(s,d) \leq \dim_B^{\rho_n}(s,d)$ for all $s < f(n+1)$ and
$d \geq f(n)$. Hence $\sdim_B^{\rho_f}(s) \leq \sdim_B^{\rho_n}(s)$ for all $s <
f(n+1)$. Thus, since $f$ is sufficiently fast growing and $\sdim_B^{\rho_n}(s)$
has subexponential growth for every $n$,
$\sdim_B^{\rho_f}$ has subexponential Borel dimension growth.

Hence by Conjecture F.(2), the orbit equivalence relation of the action $\Gamma
\actson X$ is hyperfinite, since it is equal to $E_{\rho_f}$.
\end{proof}

\subsection{Borel F{\o}lner property A}

\label{subsec:FolnerA-conjecture}

One might hope that a stronger version of part (1) of Conjecture F is true for
all Borel extended metric spaces of uniformly bounded volume growth satisfying a
suitable amenability property. We make a more speculative conjecture that this
is true for all such spaces satisfying a Borel version of Elek and Tim\'ar's
F{\o}lner property A \cite{ET25}. We recall the definition of this property:

\begin{defn}
Say that an extended metric space $(X,\rho)$ of uniformly bounded volume
growth has \define{F{\o}lner property A} if for every
$r, \ve > 0$, there is an $R > 0$ and a
function $\Theta \colon X \times X \to [0,1]$ such that 
\begin{enumerate}
\item For every $x \in X$, $y \notin B_R(x)$ implies $\Theta(x,y) = 0$, and 
$\sum_{y \in X} \Theta(x,y) = 1$.
\item If $\rho(x,x') \leq r$, then $\sum_{y \in X} |\Theta(x,y) - \Theta(x',y)| < \ve$.
\item For all $x \in X$, $\sum_{y \in X} \sum_{z \in B_r(y)} |\Theta(x,y) - \Theta(x,z)| <
\ve$.
\end{enumerate}
If $X$ is a locally finite Borel extended metric space, we say that $X$ has
\define{Borel F{\o}lner property A} if there are Borel functions $\Theta$ as above for
each $r, \ve > 0$.
\end{defn}

Note items (1) and (2) above just give the usual definition of property A. That
is, defining $\Theta_x(y) = \Theta(x,y)$, we have that $\Theta_x$ is a
probability measure supported on $B_R(x)$, and if $x$, $x'$ are distance at most
$r$, then the distance between $\Theta_x$ and $\Theta_{x'}$ is less than $\ve$ in
the $1$-norm. Item (3) is the additional F{\o}lnerness condition: 
the scale-$r$ $\ell^1$ gradient of $\Theta$: $\mathcal{E}_r(\Theta_x) = \sum_{y \in X} \sum_{z \in B_r(y)} |\Theta_x(y) -
\Theta_x(z)|$ is at most $\ve$, so
the values of this probability measure $\Theta_x$
do not change much on average when moving distance at most $r$ in the space. 

The word metric associated to any Borel action of a countable amenable group
has F{\o}lner property A \cite[Proposition 7.4]{ET25}. That is, suppose $\Gamma
\actson X$ is a Borel action
of a countable amenable group with finite symmetric generating set $S$. Let
$\rho$ be the word metric where $\rho(x,x')$ is the least $n$ so that there
exists $s_1, \ldots, s_n \in S$ such that $s_1 s_2\cdots s_n \cdot x = x'$. Then
$(X,\rho)$ has Borel F{\o}lner property A. To see this,
given $\ve, r > 0$, using the existence of F{\o}lner
sets in $\Gamma$, let $K = \{s_1 \cdots s_k \colon k \leq r \land s_i \in S\}$
be all group elements of word length at most $r$, and let $F \subset \Gamma$ be a
two-sided F{\o}lner set such that
$\frac{1}{|F|}\sum_{\gamma \in K}
  \left(|\gamma F \symdiff F| + |F\gamma \symdiff F|\right) < \ve$. 
Then defining $\Theta(x,y) = \frac{|\{\gamma \in F \colon \gamma
\cdot x = y\}|}{|F|}$, it is easy to check that $\Theta$ is as required. 
\trivial{Write $\nu_A$ for the uniform probability measure on a nonempty
finite set $A \subseteq \Gamma$, and let $\pi_x(\gamma)=\gamma\cdot x$. Then
$\Theta_x=(\pi_x)_*\nu_F$. If
$R=\max\{1,\max_{\gamma\in F}|\gamma|_S\}$, then $\Theta_x$ is a probability measure
supported on $B_R(x)$, which proves (1). If $\rho(x,x')\leq r$, choose
$\delta\in K$ with $x'=\delta\cdot x$. Since
$\Theta_{x'}=(\pi_x)_*\nu_{F\delta}$ and pushforward does not increase
$\ell^1$-distance,
\[
  \sum_{y\in X}|\Theta(x,y)-\Theta(x',y)|
  \leq \frac{|F\mathbin{\triangle}F\delta|}{|F|}<\ve.
\]
This proves (2). Finally, every $z\in B_r(y)$ equals $\delta\cdot y$ for some
$\delta\in K$, and the function $y\mapsto\Theta_x(\delta\cdot y)$ is
$(\pi_x)_*\nu_{\delta^{-1}F}$. Since $K=K^{-1}$,
\[
  \sum_{y\in X}\sum_{z\in B_r(y)}|\Theta(x,y)-\Theta(x,z)|
  \leq \frac{1}{|F|}\sum_{\delta\in K}
  |F\mathbin{\triangle}\delta^{-1}F|<\ve,
\]
which proves (3).}

Hence, the following conjecture generalizes part (1) of Conjecture F. By a graph
$G$ having Borel F{\o}lner property A or subexponential dimension growth below,
we mean that the path length metric on $G$ has these properties.

\begin{conj}\label{conj-graph-Folner-A}
Suppose $G$ is a bounded degree Borel graph on a standard Borel space $X$. Then
if $G$ has Borel F{\o}lner property A, then $G$ has subexponential Borel
dimension growth.
\end{conj}

We remark that a generalization of the above conjecture to all Borel extended
metric spaces of uniformly bounded volume growth is not true. For example, let
$(n_i)_{i \geq 1}$ be a sufficiently fast growing sequence of natural numbers,
and let $X_i$ be the metric space $\Z^{n_i}$ with the metric $\rho_i(g,h) = i \|
g-h\|_\infty$ of
$i$ multiplied by the usual $\ell_\infty$ metric. Then consider the disjoint
union $X = \bigsqcup_i X_i$ of all these spaces, with the metric where $\rho(x,x') =
\rho_i(x,x')$ if $x,x' \in X_i$ for some $i$, and $\rho(x,x') = \infty$
otherwise. Note that $X$ has uniformly bounded volume growth.
Since $\sdim^{\Z^n}(1) = n$, it follows that $\sdim^X(i) \geq n_i$, so the
dimension growth of $X$ can be arbitrarily fast (and hence not subexponential), but it is easy to show that this
space has F{\o}lner property $A$. Since
these spaces are countable, the Borel and classical dimension agree.

This counterexample reflects the fact that the right setting here is perhaps that
of Borel coarse geometry: given a metric $\rho$, the key object to analyze is
really the metric generated by walks of steps of length at most $s$. That is,
the associated metric $\rho_s$ where $\rho_s(x,x')$ is the infimum length of a
$\rho$-path from $x$ to $x'$ where each step in the path has distance at most
$s$. (Or in the language of coarse geometry, the relations of having distance at
most $s$ for each $s$ generate a natural collection of the
entourages for $\rho$). 
The correct metric generalization of \Cref{conj-graph-Folner-A} is that if $\rho$ is a
Borel extended metric with uniformly bounded volume growth and Borel F{\o}lner
property A, then the metric $\rho_s$ has subexponential Borel dimension growth
for every $s$. Likewise, a construction similar to that in \ref{conjF-implies-Weiss} shows
that if the entourages of $\rho$ all have subexponential dimension growth, then
one can similarly make a stretched version $\rho_f$ of the metric keeping the
same connectedness relation so that $\rho_f$ has subexponential Borel dimension
growth.

\printbibliography

\end{document}